\documentclass[a4paper,12pt]{amsart}

\usepackage{amsmath,amssymb,amsthm}
\usepackage[all]{xy}
\usepackage{hyperref}

\newtheorem{thrm}{Theorem}[section]
\newtheorem{prop}[thrm]{Proposition}
\newtheorem{coro}[thrm]{Corollary}
\newtheorem{lemm}[thrm]{Lemma}
\newtheorem{maintheorem}{Theorem}

\theoremstyle{definition}
\newtheorem{defn}[thrm]{Definition}
\newtheorem{exam}[thrm]{Example}

\newtheorem{rema}[thrm]{Remark}

\renewcommand{\bar}[1]{\overline{#1}}
\renewcommand{\hat}[1]{\widehat{#1}}
\renewcommand{\tilde}[1]{\widetilde{#1}}
\renewcommand{\epsilon}{\varepsilon}

\newcommand{\Z}{\mathbb{Z}}

\newcommand{\R}{\mathbb{R}}
\newcommand{\C}{\mathbb{C}}

\newcommand{\G}{\mathbb{G}}
\newcommand{\K}{\mathbb{K}}
\renewcommand{\H}{\mathbb{H}}
\newcommand{\eps}{\varepsilon}
\renewcommand{\phi}{\varphi}
\newcommand{\map}[3]{#1\colon#2 \longrightarrow #3}
\newcommand{\act}[3]{#1\colon#2 \curvearrowright #3}
\newcommand{\abs}[1]{\lvert #1 \rvert}
\newcommand{\nor}[1]{\lVert #1 \rVert}
\newcommand{\ip}[1]{\left\langle #1 \right\rangle}

\newcommand{\cl}[1]{\mathcal{#1}}

\newcommand{\mbf}[1]{\mathbf{#1}}
\newcommand{\ind}[1]{\mathrm{Ind}\,#1}
\newcommand{\obj}[1]{\mathrm{Obj}\,#1}
\newcommand{\tensor}{\otimes}
\newcommand{\fin}{\mathrm{f}}
\renewcommand{\star}{$\ast$}
\newcommand{\cstar}{C*}
\newcommand{\tend}{\textendash}

\let\mod\relax

\DeclareMathOperator{\id}{\mathrm{id}}

\DeclareMathOperator{\tr}{\mathrm{Tr}}
\DeclareMathOperator{\Index}{\mathrm{Index}}
\DeclareMathOperator{\cspan}{\bar{\mathrm{span}}}
\DeclareMathOperator{\aspan}{\mathrm{span}}
\DeclareMathOperator{\hilb}{\mathrm{Hilb}}
\DeclareMathOperator{\rep}{\mathrm{Rep}}
\DeclareMathOperator{\mod}{\mathrm{Mod}}
\DeclareMathOperator{\corr}{\mathrm{Corr}}
\DeclareMathOperator{\rfin}{\mathrm{rf}}
\DeclareMathOperator{\irr}{\mathrm{Irr}}

\DeclareMathOperator{\End}{\mathrm{End}}
\DeclareMathOperator{\picard}{\mathrm{Pic}}
\DeclareMathOperator{\auto}{\mathrm{Aut}}
\DeclareMathOperator{\ev}{\mathrm{Ev}}

\title{Equivariant covering spaces of quantum homogeneous spaces}
\author{Mao Hoshino}
\address{Department of Mathematical Sciences, The University of Tokyo\\
Komaba 3-8-1, Tokyo \mbox{153-8914}, Japan}
\email{mhoshino@ms.u-tokyo.ac.jp}
\subjclass[2020]{Primary~46L67, Secondary~46L08, 17B37}
\keywords{operator algebra, quantum group, tensor category}

\begin{document}

\begin{abstract}
We develop a fundamental theory of compact quantum group
equivariant finite extensions of \cstar-algebras. In particular we 
focus on the case of quantum homogeneous spaces and give a 
Tannaka-Krein type result for equivariant correspondences.
As its application, we show that every Jones' value appears as
the index of an equivariant conditional expectation.
In the latter half of this paper, we give an imprimitivity 
theorem in some cases:
for general compact quantum groups under a finiteness conditions, and
for the Drinfeld-Jimbo deformation $G_q$ of 
a simply-connected compact Lie group $G$.
As an application, we give a complete classification of finite index
discrete quantum subgroups of $\hat{G_q}$.
\end{abstract}

\maketitle

\section{Introduction}

The notion of a \cstar-tensor category has played an important role
in recent developments of theories of operator algebras. It firstly
appeared in the subfactor theory, as a tool of classification 
of subfactors: Actually every inclusion of factors gives a \cstar-tensor
category as an invariant, which is complete in the amenable case ({\cite[Remark 7.2.1]{Po94}}).
Moreover we also can use it to understand V. F. Jones' celebrated work 
on the range of indices of subfactors ({\cite[Theorem 4.3.1]{Jo83}}).

\cstar-tensor categories also appear in the theory of compact 
quantum groups, as their representation categories. Moreover
the notion of a module category over a \cstar-tensor category is
also useful. By using this, the classical Tannaka-Krein duality is 
generalized not only to a compact quantum group itself, 
but also to its action on \cstar-algebra ({\cite[Theorem 6.4]{DY13}}). 
De Commer and Yamashita use this duality theorem to obtain 
a one-to-one correspondence between quantum homogeneous spaces 
of $SU_q(2)$ and concrete combinatorial datum (\cite[Theorem 2.4]{DY15}). 

The purpose of this paper is to connect these \cstar-tensor categorical
approaches and study an inclusion of \cstar-algebras with actions of
a compact quantum group. In this paper, we call it as an \emph{equivariant
finite quantum covering spaces}, since it can be regarded as a genuine
finite covering space in the case of commutative \cstar-algebras. 
We also call
its minimal indices as its \emph{covering degree}. As in the 
non-equivariant case, an equivariant finite quantum covering space can be
understood as a Q-system in a suitable \cstar-tensor category,
namely the category of equivariant correspondences. The following
theorem on this category, which is a generalization of
the observation due to De Commer and Yamashita ({\cite[Theorem 7.1]{DY13}}),
is fundamental throughout this paper.

\begin{maintheorem}[Theorem \ref{thrm:fundamental theorem}]
For quantum homogeneous spaces $A$ and $B$ of a compact quantum group $\G$,
we have the following canonical equivalence of \cstar-categories:
\[
 \G\text{-}\corr^{\rfin}_{A,B} \cong [\G\text{-}\mod^{\fin}_A,\G\text{-}\mod^{\fin}_B]^{\rep^{\fin}\G}.
\]
\end{maintheorem}

The first application of this result is the following existence theorem.

\begin{maintheorem}[Theorem \ref{thrm:Jones value}]
For any $d \in \{4\cos^2 (\pi/n)\mid n \geq 3\}\cup[4,\infty)$,
we have a compact quantum group $\G$ and a finite quantum $\G$-covering 
space $A \subset B$ with its covering degree $d$.
\end{maintheorem}

In Section 4, we focus on the $\G$-actions induced from the maximal Kac 
quantum subgroup $\K$ of $\G$. Instead of treating with them directly, 
we work on $\rep^{\fin}\G$-module categories and module functors 
between them. This enables us to translate the imprimitivity theorem 
to the comparison theorem of $\rep^{\fin}\G$-module functors 
and $\rep^{\fin}\K$-module functors. Then we develop a general theory 
of module categories admitting a \emph{module trace}, concluding
the following theorem for actions of compact quantum groups.

\begin{maintheorem}[Theorem \ref{thrm:main1}, Theorem \ref{thrm:picard}, Theorem \ref{thrm:auto}]
Let $A$ be a quantum homogeneous space of $\K$ and $\tilde{A}$ be 
its induced $\G$-\cstar algebra. 
\begin{enumerate}
 \item If $A$ has a tracial state and $\irr\K\text{-}\mod^{\fin}_A$ is 
finite, the induction functor gives an equivalence of 
\cstar-tensor categories: 
$\K\text{-}\corr^{\rfin}_A\cong \G\text{-}\corr^{\rfin}_{\tilde{A}}$. 
Moreover both of them are rigid.
 \item If $A$ is induced from a quantum homogeneous space of a cocommutative
quantum subgroup of $\K$, the induction functor gives group isomorphisms
$\picard_{\K}(A)\cong \picard_{\G}(\tilde{A})$ and $\auto_{\K}(A)\cong \auto_{\G}(\tilde{A})$.
\end{enumerate}
\end{maintheorem}

The statement (i) contains the proceeding research \cite{CKS,So05} on 
the quantum Bohr compactification as a special case.

In Section 5, we deal with the Drinfeld-Jimbo deformation.
By using the well-developed
representation theory of $C(G_q)$ and $C(T\backslash G_q)$,
we show the following theorems.

\begin{maintheorem}[Corollary \ref{coro:descent for G_q}]
Let $A,B$ be quantum homogeneous spaces of $T$ and 
$\tilde{A},\tilde{B}$ be its induced $G_q$-\cstar algebras. Then 
the induction functor gives an equivalence of \cstar-categories:
$T\text{-}\corr^{\rfin}_{A,B}\cong G_q\text{-}\corr^{\rfin}_{\tilde{A},\tilde{B}}$
\end{maintheorem}
\begin{maintheorem}[Corollary \ref{coro:automatically induced}]
Let $\tilde{A}$ be a quantum homogeneous space
containing $C(T\backslash G_q)$ and admitting a tracial state.
Then it is of the form $\ind_T^{G_q} A$, where $A$ is 
a quantum homogeneous space of $T$. 
\end{maintheorem}

At the last of this paper, we use the result in Section 4
to give a complete classification result of
finite index discrete quantum subgroups of $\hat{G_q}$.

\begin{maintheorem}[Theorem \ref{thrm:classification of subgroup}]
Let $P$ (resp. Q) be the weight (resp. root) lattice of $G$.
There is a canonical one-to-one correspondece between finite index
discrete quantum subgroups of $\hat{G_q}$ and 
subgroups of $P/Q$.
\end{maintheorem}

\section{Preliminaries}

In this paper, the symbol $\textendash\tensor\textendash$ denotes
the spatial tensor product of C*-algebras, the algebraic tensor
product of $\C$-vector spaces and the external tensor product of 
Hilbert C*-modules.

For a Hilbert space $\cl{H}$, its inner product $\ip{\tend,\tend}$ is 
$\C$-linear with respect to the second argument. An element $\xi \in \cl{H}$
is considered as an operator from $\C$ to $\cl{H}$. If $\cl{H}$ is 
finite dimensional, the unnormalized trace on $B(\cl{H})$ is 
denoted by $\tr$.

\subsection{Compact quantum groups and their dual}
In this subsection, we give a brief review on compact quantum groups.
See \cite{NT13} for detailed discussions.

A \emph{compact quantum group} is a pair $\G = (A,\Delta)$
of a unital C*-algebra $A$ and \star-homomorphism 
$\map{\Delta}{A}{A\otimes A}$ which satisfies the following conditions: 
\begin{itemize}
 \item the coassociativity $(\Delta\tensor\id)\Delta = (\id\tensor\Delta)\Delta$, 
 \item the cancellation property $\bar{(A\otimes \C)\Delta(A)} = \bar{(\C\tensor A)\Delta(A)} = A\tensor A$.
\end{itemize}
In this case $A$ is denoted by $C(\G)$. A \emph{Haar state} on $\G$ is a
state $h$ on $C(\G)$ satisfying the bi-invariance property 
$(h\tensor\id)\Delta(x) = (\id\tensor h)\Delta(x) = h(x)1_{\G}$. Such a 
state always exists and is unique. It is
said that $\G$ is \emph{reduced} when $h$ is faithful. 
We only consider reduced compact quantum groups unless otherwise noted.

A \emph{unitary representation} of $\G$ is a pair 
$\pi = (\cl{H}_{\pi},U_{\pi})$ of a Hilbert space $\cl{H}_{\pi}$ and 
a unitary $U_{\pi} \in M(\cl{K}(\cl{H}_{\pi}) \otimes C(\G))$ 
which satisfies $(\id\otimes\Delta)(U_{\pi}) = U_{\pi,12}U_{\pi,13}$. 
The category of unitary representations of $\G$ is denoted by $\rep \G$, 
and its full subcategory of the finite dimensinal ones is denoted by 
$\rep^{\fin} \G$.

Let $\pi$ be a finite dimensional unitary representation of $\G$ and 
$\xi,\,\eta$ be elements of $\cl{H}_{\pi}$. Then 
$(\xi^*\otimes 1)U_{\pi}(\eta\tensor 1)$ defines an element of 
$C(\G)$, called a \emph{matrix coefficient} of $\pi$. The set of all 
matrix coefficients of all finite dimensional unitary representations is 
called the \emph{algebraic core} of $\G$ and denoted by $\cl{O}(\G)$. 
We can make $\cl{O}(\G)$ into a Hopf \star-algebra by using the product and 
the coproduct of $C(\G)$. Its counit and antipode are denoted by 
$\eps$ and $S$, respectively. In general $S$ satisfies $S(S(x)^*)^* = x$ 
for any $x \in \cl{O}(\G)$, but does not 
$S^2 = \id$. It is said that $\G$ is \emph{of Kac type} when it holds.

Let $\H$ be another compact quantum group. A \emph{homomorphism} 
from $\H$ to $\G$ is a  homomorphism of Hopf \star-algebras from
$\cl{O}(\G)$ to $\cl{O}(\H)$. One should note that
$\phi$ need not be defined on $C(\G)$. On the other hand 
we can still define $(\phi\tensor \id)\Delta$ as a \star-homomorphism from
$C(\G)$ to $C(\H)\tensor C(\G)$. More generally, we have the following 
lemma, which is a quantum analogue of the Fell's absorption principle
for $\hat{\G}$:

\begin{lemm} \label{lemm:absorption}
 Let $A$ be a \cstar-algebra and $\map{\phi}{\cl{O}(\G)}{A}$ be a 
\star-homomorphism. Then $(\phi\tensor\id)\Delta$ extends to a 
\star-homomorphism from $C(\G)$ to $A\tensor C(\G)$.
\end{lemm}
\begin{proof}
Let $(\pi,L^2(\G))$ be the GNS representation of $C(\G)$ with respect to
the Haar state of $\G$. Since $\G$ is reduced, this representation is 
faithful.

It suffices to show in the case of $A = B(\cl{H})$ for a Hilbert space
$\cl{H}$. Let $V$ be an operator on $\cl{H}\tensor L^2(\G)$ satisfying 
\[
 V(\xi\tensor \Lambda(x)) 
= (\phi\tensor\pi)(\Delta(x))(\xi\tensor \Lambda(1)).
\]
Then $V$ is unitary and satisfies 
$V(1\tensor \pi(x))V^* = (\phi\tensor\pi)(\Delta(x))$. 
Then the statement follows from the faithfulness of $\pi$.
\end{proof}

If we are given a Hopf \star-algebra $\cl{A}$ generated by 
matrix coefficients of its finite dimensional unitary representations, 
we can construct a compact quantum group whose algebraic core
coincides with $\cl{A}$. This fact allows us to construct the 
maximal Kac quantum subgroup.

\begin{defn}
The \emph{maximal Kac quantum subgroup} of $\G$ is the compact quantum group
corresponding to a Hopf \star-algebra obtained as the quotient of 
$\cl{O}(\G)$ divided by a two-sided ideal generated with 
$S^2(x) - x$ for all $x \in \cl{O}(\G)$.
\end{defn}

\begin{rema}
Actually $\K$ is of Kac type and has the following universal property: 
Let $\H$ be a compact quantum group of Kac type. Then any homorphism from
$\H$ to $\G$ factors through $\K$. This follows from that
any homomorphism of Hopf \star-algebra preserves antipodes. 

\end{rema}
\begin{rema} \label{rema:canonical Kac}
Let $\G_u$ be the universal form of $\G$ and $\tilde{\K}$ be its 
canonical Kac subgroup, introduced in {\cite[Appendix]{So05}}. 
Then we have a canonical \star-homomorphism
$\map{\tilde{q}}{\cl{O}(\G)}{\cl{O}(\tilde{\K})}$. This map satisfies
$\tilde{q}\circ S^2 = \tilde{q}$, hence $\tilde{q}$ 
factors through $\cl{O}(\K)$. On the other hand, the universal property of 
$\G_u$ implies that a canonical map $\map{q}{\cl{O}(\G)}{\cl{O}(\K)}$
extends to \star-homomorphism from $C_u(\G)$ to $C_r(\K)$. Since the Haar
state of $\K$ is faithful on $C_r(\K)$, this map factors through
$C(\tilde{\K})$ by its construction. This implies that $q$ factors through
$\cl{O}(\tilde{\K})$ and gives an isomorphism from 
$\cl{O}(\tilde{\K})$ to $\cl{O}(\K)$.
\end{rema}

% As an application of the above remark, we have the following lemma.

% \begin{lemm}
%  Let $\tau$ be a tracial linear functional on $\cl{O}(\G)$. If 
% $\tau$ is positive in the sense that $\tau(x^*x) \ge 0$ holds for any
% $x \in \cl{O}(\G)$, it factors through $\cl{O}(\K)$.
% \end{lemm}
% \begin{proof}
% By Remark \ref{rema:canonical Kac}, we can replace $\cl{O}(\K)$ 
% by $\cl{O}(\tilde{\K})$. Then our statement follows from the definition
% of $\tilde{\K}$, since $\tau$ defines a tracial positive linear functional
% on $C_u(\G)$.
% \end{proof}

\begin{exam}
Let $G$ be a semisimple compact Lie group and $q \in (-1,1)$. Then there
is a compact quantum group $G_q$ called the \emph{Drinfeld-Jimbo
$q$-deformation} of $G$, see {\cite[Definition 2.4.5.]{NT13}} for
a precise definition.
The representation theory of $G_q$ is quite 
similar to that of $G$, which means that they have the common set of 
equivalence classes of irreducible representations and a common fusion rule.
But actually $\rep^{\fin}{G_q}$ is not equivalent to $\rep^{\fin}{G}$
as a C*-tensor category.

It is shown that the maximal Kac quantum subgroup of $G_q$ coincides with 
the maximal torus $T$ of $G$ ({\cite[Lemma 4.10.]{To07}}).
\end{exam}

We end this subsection with the notion of \emph{discrete quantum group}.
Its definition is given by the Pontrjagin dual of a compact quantum group,
hence the reduced \cstar-algebra of discrete quantum group $\Gamma$ is 
nothing but $C(\G)$ when $\Gamma$ is the Pontrjagin dual of $\G$.
In this case $\Gamma$ is denoted by $\hat{\G}$. 

For the notion of closed quantum subgroup, see \cite{DKSS}.
There are several approaches, but all of them are equivalent in 
the discrete case ({\cite[Theorem 6.2]{DKSS}}). As a result, we can 
take the following definition:
 a \emph{closed quantum subgroup} 
of $\G$ is a pair of discrete quantum group $\hat{\H}$ 
and a coproduct-preserving unital inclusion $C(\H) \subset C(\G)$. 
In this case a unitary representation of $\H$ gives a unitary repsentation
of $\G$ with the same intetwiner space.
Combining this with the Woronowicz's Tannaka-Krein duality 
({\cite[Theorem 2.3.2]{NT13}}), we have
a one-to-one correspondence between discrete quantum subgroups of 
$\hat{\G}$ and \cstar-tensor full subcategories of $\rep^{\fin}\G$.

\subsection{Actions of compact quantum groups}

For detailed discussions and references, see \cite{DC16}.
Let $\G$ be a compact quantum group. A (reduced) \emph{$\G$-C*-algebra} or
\emph{$\G$-action} on a C*-algebra is a pair $(A,\alpha)$ of a C*-algebra 
$A$ and faithful \star-homomorphism $\map{\alpha}{A}{A\otimes C(\G)}$ 
with the following conditions: 
\begin{itemize}
 \item $(\id\tensor\Delta)\alpha = (\alpha\tensor\id)\alpha$, 
 \item $\cspan{(\C\tensor C(\G))\alpha(A)} = A\tensor C(\G)$. 
\end{itemize}
Its fixed point subalgebra $\{x \in A\mid \alpha(x) = x\otimes 1\}$ is 
denoted by $A^{\G}$ or $A^{\alpha}$. We also define 
the algebraic core of $A$ as 
$\cl{A} = \{x \in A\mid \alpha(x) \in A\otimes \cl{O}(\G)\}$.
Then we have a right coaction of the Hopf \star-algebra $\cl{O}(\G)$ on 
$\cl{A}$ given by the restriction of $\alpha$ to $\cl{A}$.

We say that $\G$-\cstar-algebra $A$ is 
a \emph{quantum homogeneous space} of $\G$ if it is unital and 
$A^{\G} = \C1_A$.

\begin{exam}
The C*-algebra $C(\G)$ has a canonical $\G$-action defined by $\Delta$. 
Moreover, every $\G$-invariant C*-subalgebra of $C(\G)$ also
has a $\G$-action and actually is a quantum homogeneous space of $\G$.
We call it a \emph{right coideal} of $\G$.
\end{exam}

\begin{exam}
Let $\H$ be a compact quantum subgroup of $\G$ and 
$\map{q}{\cl{O}(\G)}{\cl{O}(\H)}$ be the canonical map. Then 
$(q\tensor\id)\Delta$ extends to a \star-homomorphism $\ell_{\H}$ on 
$C(\G)$ by Lemma \ref{lemm:absorption}. Since $\ell_{\H}$ is 
$\G$-equivariant with respect to the canonical $\G$-action on $C(\G)$, 
we have a right coideal 
$C(\H\backslash\G) = \{x \in C(\G)\mid \ell_{\H}(x) = 1\otimes x\}$.

More generally, for a $\H$-\cstar-algebra, 
the \emph{induced $\G$-\cstar-algebta} of $B$ 
is defined as a pair of the following:
\begin{itemize}
 \item $\ind_{\H}^{\G} A = \{x \in A\otimes C(\G)\mid (\alpha\otimes \id)(x) = (\id\otimes\ell_{\H})(x)\}$,
 \item $\tilde{\alpha} = $ the restriction of $\id\tensor\Delta$.
\end{itemize}
For example, the induced $\G$-\cstar-algebra of the trivial 
$\H$-C*-algebra $\C$ is isomorphic to $C(\H\backslash\G)$.
\end{exam}

If $A$ and $B$ are $\G$-\cstar-algebras with the algebraic cores $\cl{A}$ 
and $\cl{B}$ respectively, we say that a linear map 
$\map{\phi}{\cl{A}}{B}$ is $\G$-equivariant when 
$(\phi\tensor\id)\alpha = \beta\phi$ holds. Similarly the equivariance for
$\map{\psi}{A}{B}$ is also defined when 
$\map{\psi\tensor\id}{A\tensor C(\G)}{B\tensor C(\G)}$ is defined. 

The following generalization of Lemma \ref{lemm:absorption} is useful
to get a $\G$-equivariant c.p.~map. If $\cl{A}$ is \star-algebra
and $B$ is \cstar-algebra, we say that a linear map 
$\map{\phi}{\cl{A}}{B}$ is completely positive if $(\phi\tensor\id)(X^*X)$
is positive in $B\tensor M_n(\C)$ for any $X \in \cl{A}\tensor M_n(\C)$.

\begin{prop} \label{prop:automatic continuity}
Let $A$ and $B$ be $\G$-\cstar-algebras with the algebraic cores
$\cl{A}$ and $\cl{B}$ respectively. If $\map{\phi}{\cl{A}}{B}$ is 
$\G$-equivariant and completely positive, 
it extends to a $\G$-equivariant c.p.~map from $A$ to $B$.
\end{prop}
For a proof of this proposition, see Subsection \ref{subsection:A1}.

Next we move on the notion of modules over $\G$-\cstar-algebras.

\begin{defn}
Let $A$ be a $\G$-C*-algebra. A \emph{$\G$-equivariant Hilbert $A$-module} 
is a pair $(E,\alpha_E)$ of a Hilbert $A$-module $E$ and a linear 
map $\map{\alpha_E}{E}{E\otimes C(\G)}$ with the following conditions:
\begin{itemize}
 \item $\alpha(\ip{\xi,\eta}_A) = \ip{\alpha_E(\xi),\alpha_E(\eta)}_{A\tensor C(\G)}$ for any $\xi,\,\eta \in E$,
 \item $\cspan{\alpha_E(E)(\C\tensor C(\G))} = E\otimes C(\G)$.
\end{itemize}
Here $E\otimes C(\G)$ means the external tensor product of the Hilbert 
$A$-module $E$ and the Hilbert $C(\G)$-module $C(\G)$. 
\end{defn}

If $A$ is a quantum homogeneous space and $E$ is finitely generated, 
$E$ must be projective ({\cite[Theorem 6.21.]{DC16}}). 

The category of $\G$-equivariant Hilbert $A$-module is denoted 
by $\G\text{-}\mod_A$, and its full subcategory of finitely 
generated ones is denoted by $\G\text{-}\mod^{\fin}_A$. We use
the symbol $\cl{L}_A(\tend,\tend)$ for the spaces of 
adjointable right $A$-module maps. We also use 
$\cl{L}_A^{\G}(\tend,\tend)$ for the spaces of $\G$-equvariant ones.

For $E \in \G\text{-}\mod^{\fin}_A$, we have a $\G$-action 
$\act{\tilde{\alpha}}{\G}{\cl{K}_A(E)}$ on the algebra of compact $A$-module
maps. Moreover we can extend this action to a \star-homomorphism $\map{\tilde{\alpha}}{\cl{L}_A(E)}{\cl{L}_{A\tensor C(\G)}(E\tensor C(\G))}$,
and this satisfies $(\tilde{\alpha}\tensor\id)\tilde{\alpha} = (\id\tensor \Delta)\tilde{\alpha}$. For any $T \in \cl{L}_A(E)$, the operator $\tilde{\alpha}(T)$ is the only operator on $E\tensor C(\G)$ satisfying 
$\tilde{\alpha}(T)\alpha_E(\xi) = \alpha_E(T\xi)$.

We can also define the algebraic core and the induction of equivariant 
Hilbert C*-modules as follows:
\begin{itemize}
 \item $\cl{E} = \{\xi \in E\mid \alpha_E(\xi) \in E\tensor \cl{O}(\G)\}$,
 \item $\ind_{\H}^{\G}F = \{x \in F\tensor C(\G)\mid (\alpha_F\otimes \id)(x) = (\id\tensor\ell_{\H})(x)\}$.
\end{itemize}
It can be easily seen that $\cl{E}$ has a right $\cl{A}$-action and 
a right coaction of $\cl{O}(\G)$.

\subsection{C*-tensor categories and their module categories}
For detailed descriptions, see \cite{NT13}.
A \emph{\cstar-category} is a $\C$-linear category with a norm on 
each Hom space and an anti-linear involution 
$\map{\ast}{\cl{C}(X,Y)}{\cl{C}(Y,X)}$ satisfying 
the \cstar-identity: 
$\nor{T^*T} = \nor{T}^2$. In this paper, we also assume that 
\cstar-categories are closed under taking direct sums and 
subobjects.

Let $\cl{C}$ and $\cl{D}$ be \cstar-categories. A functor 
$\map{F}{\cl{C}}{\cl{D}}$ is called a \emph{\cstar-functor} if it preserves
the adjoints, i.e. $F(T^*) = F(T)^*$ for any morphism $T$. 
For a natural transformation between \cstar-functors, we can define its
adjoint by taking the adjoint of each components. We use the symbol
$[\cl{C},\cl{D}]_{\mathrm{b}}$ to denotes the category of 
\cstar-functors from $\cl{C}$ to $\cl{D}$ and bounded natural 
transformations. Here the boundedness of a natural transformation is 
defined by the uniform norm-boudedness of its components. Obviously 
$[\cl{C},\cl{D}]_{\mathrm{b}}$ also has a structure of C*-category.

A \emph{strict \cstar-multitensor category} is a triple $(\cl{C},\tensor,\mbf{1})$ of \cstar-category 
$\cl{C}$, a \cstar-bifunctor 
$\map{\tend\tensor\tend}{\cl{C}\times\cl{C}}{\cl{C}}$
and an object $\mbf{1}$ called the unit object which satisfies 
$U\tensor (V\tensor W) = (U\tensor V)\tensor W$ and
$\mbf{1}\tensor U = U = U\tensor\mbf{1}$ for any objects 
$U,V,W \in \cl{C}$. If $\cl{C}(\mbf{1},\mbf{1}) = \C\id_{\mbf{1}}$ holds, 
$\cl{C}$ is said to be a \emph{strict \cstar-tensor category}. 
Let $X$ be an object of $\cl{C}$. We say that $X$ is \emph{rigid} 
if there is a quadruple $(X,\bar{X},R,\bar{R})$
with an object $\bar{X}$ of $\cl{C}$, 
$R \in \cl{C}(\mbf{1},\bar{X}\tensor X)$ and 
$\bar{R} \in \cl{C}(\mbf{1},X\tensor\bar{X})$ which satisfies the 
\emph{conjugate equations}
\[
(\id_X\tensor R^*)(\bar{R}\tensor \id_X) = \id_X,\quad 
(\id_{\bar{X}}\tensor \bar{R}^*)(R\tensor \id_{\bar{X}}) = \id_{\bar{X}}.
\]
This quadruple is called a \emph{solution of the conjugate equations}.
The full subcategory of $\cl{C}$ consisting of all rigid objets of $\cl{C}$
is denoted by $\cl{C}^{\fin}$, and $\cl{C}$ is said to be a \emph{rigid}
when $\cl{C} = \cl{C}^{\fin}$, i.e. all objects of $\cl{C}$ are rigid.
In such a \cstar-tensor category, each Hom space is finite dimensional
({\cite[Proposition 2.2.8]{NT13}}).

Let $X$ be a rigid object of a strict rigid \cstar-tensor category $\cl{C}$. Then we can minimize the value 
$\nor{R}\nor{\bar{R}}$ for a solution of 
conjugate equation. This minimum value is called the categorical dimension 
of $X$, denoted by $d(X)$.  A solution $(X,\bar{X},R,\bar{R})$ is said to 
be \emph{standard} when $\nor{R}^2 = \nor{\bar{R}}^2 = d(X)$. 
We remark here that any two standard solutions with fixed $X$ are 
unitary equivalent to each other, which means that 
we can get all standard solutions by considering 
$(X, \bar{X}', (u\otimes \id)R, (\id\tensor u)\bar{R})$ with a 
fixed standard solution $(X,\bar{X},R,\bar{R})$ and an arbitrary unitary 
$u \in \cl{C}(\bar{X},\bar{X}')$. This fact enables us to define the 
categorical trace $\map{\tr_X}{\cl{C}(X,X)}{\C}$ given by
the formula $\tr_X(T)\id_{\mbf{1}} = R^*(\id\tensor T)R$ using a standard
solution. This is independent of the choice of standard solutions and 
equal to $\bar{R}^*(T\tensor\id)\bar{R}$. In a similar way, 
we can also define the partial traces 
$\map{\tr_X\tensor \id}{\cl{C}(X\tensor Y,X\tensor Z)}{\cl{C}(Y,Z)}$
and $\map{\id\tensor\tr_X}{\cl{C}(Y\tensor X,Z\tensor X)}{\cl{C}(Y,Z)}$.
The categorical traces are actually tracial in the following sense: For 
any morphisms $S,T \in \cl{C}(X,Y)$ we have $\tr_X(S^*T) = \tr_Y(TS^*)$.

Let $\cl{C,\,D}$ be strict C*-multitensor categories. A C*-tensor
functor from $\cl{C}$ to $\cl{D}$ is a pair $(\Theta,\theta)$ of 
a \cstar-functor 
$\map{\Theta}{\cl{C}}{\cl{D}}$ and a unitary natural transformation 
$\map{\theta}{\Theta(\tend)\tensor \Theta(\tend)}{\Theta(\tend\tensor\tend)}$ with
the conditions 
\begin{itemize}
 \item $\Theta(\mbf{1}_{\cl{C}}) = \mbf{1}_{\cl{D}}$,
 \item The following diagram commutes for $U,V,W \in \cl{C}$:
\[
 \xymatrix@C=40pt
{
\Theta(U)\otimes \Theta(V)\otimes \Theta(W) \ar[r]^-{\id\tensor \theta_{V,W}} \ar[d]_-{\theta_{U,V}\tensor \id} & \Theta(U) \tensor \Theta(V\tensor W) \ar[d]^-{\theta_{U,V\tensor W}}\\
\Theta(U\otimes V)\otimes \Theta(W) \ar[r]_-{\theta_{U\tensor V,W}}& \Theta(U\tensor V\tensor W).
}
\]
\end{itemize}
If $(\Theta',\theta')$ is also a C*-tensor functor from $\cl{C}$ to $\cl{D}$, then 
a natural transformation $\map{\eta}{F}{G}$ is said to be a monoidal
if it satisfies 
$\theta'_{U,V}\circ (\eta_U\tensor\eta_V) = \eta_{U\tensor V}\circ \theta_{U,V}$
for $U,V \in \cl{C}$.

\begin{defn}[{\cite[Definition 2.14.]{DY13}}]
Let $\cl{C}$ be a strict rigid \cstar-tensor category. 
A \emph{$\cl{C}$-module category} is a triple $(\cl{M},\tensor,a)$ 
of a C*-category $\cl{M}$, a C*-bifunctor 
$\map{\tend\tensor\tend}{\cl{C}\times\cl{M}}{\cl{M}}$ and a unitary 
natural transformation 
$\map{a}{\tend\tensor(\tend\tensor\tend)}{(\tend\tensor\tend)\tensor\tend}$ 
with the following conditions.
\begin{itemize}
 \item $\mbf{1}\tensor X = X$ for any $X \in \cl{M}$,
 \item The following diagram commutes for any $U,\,V,\,W \in \cl{C}$ and $X \in \cl{M}$:
\[
 \xymatrix@C=40pt
{
U\tensor (V\tensor (W\tensor X)) \ar[r]^-{a_{U,V,W\tensor X}} \ar[d]_-{\id\tensor a_{V,W,X}} & (U\tensor V)\tensor(W\tensor X) \ar[d]^-{a_{U\tensor V, W,X}} \\
U\tensor ((V\tensor W)\tensor X) \ar[r]_-{a_{U,V\tensor W,X}} & (U\tensor V\tensor W)\tensor X.
}
\]
\end{itemize}
\end{defn}
We remark that a $\cl{C}$-module structure on a fixed C*-category $\cl{M}$
gives rise to a C*-tensor functor from $\cl{C}$ to 
$[\cl{M},\cl{M}]_{\mathrm{b}}$ and vice versa.
It is said that $\cl{M}$ is semisimple if $\cl{M}(X,Y)$ is finite 
dimensional for any $X,\,Y \in \cl{M}$. We assume that all 
$\cl{C}$-module categories in this paper are semisimple. 
 A $\cl{C}$-module category $\cl{M}$ is said to be connected
if, for any $X,Y \in \cl{M}$, there is $U \in \cl{C}$ and 
an isometry from $X$ to $U\tensor Y$.
\begin{exam}
Let $\G$ be a compact quantum group. Then its representaiton category 
$\rep \G$ is a \cstar-tensor category and $\rep^{\fin}\G$ is a 
rigid \cstar-tensor category.

Take a quantum homogeneous space $A$ of $\G$. 
Then $\G\text{-}\mod_A^{\fin}$ has a structure of 
$\rep^{\fin}\G$-module category: For $\pi \in \rep^{\fin}\G$ and $E \in \G\text{-}\mod^{\fin}_A$, $\cl{H}_{\pi}\tensor E$ can be made into a 
$\G$-equivariant Hilbert $A$-module with a $\G$-action 
$v\tensor\xi \longmapsto U_{\pi,13}(v\tensor \alpha_E(\xi))$.
This $\rep^{\G}$-module category is connected and semisimple ({\cite[Proposition 3.11]{DY13}}).
\end{exam}

If $\cl{N}$ is another $\cl{C}$-module category, 
a \emph{$\cl{C}$-module functor}
 from $\cl{M}$ to $\cl{N}$ is a pair of a C*-functor 
$\map{F}{\cl{M}}{\cl{N}}$ and a unitary natural transformation 
$\map{f}{F(\tend\tensor\tend)}{\tend\tensor F(\tend)}$ with 
the commutativity of the following diagram for any $U,\,V \in \cl{C}$ and 
$X \in \cl{M}$:
\[
 \xymatrix
{
F(U\tensor (V\tensor X)) \ar[r]^-{f_{U,V\tensor X}} \ar[d]_-{F(a_{U,V,X})} & U\tensor F(V\tensor X) \ar[r]^-{\id\tensor f_{V,X}} & U\tensor (V\tensor F(X)) \ar[d]^-{a_{U,V,F(X)}}\\
F((U\tensor V) \tensor X) \ar[rr]_-{f_{U\tensor V,X}}&& (U\tensor V) \tensor F(X).
}
\]
The category of $\cl{C}$-module functors from $\cl{M}$ to $\cl{N}$ is 
denoted by $[\cl{M},\cl{N}]^{\cl{C}}_{\mathrm{b}}$, here a morphism from 
$(F,f)$ to $(G,g)$ is given by a bounded natural transformation 
$\map{\eta}{F}{G}$ making the following diagram commutative for any 
$U \in \cl{C}$ and $X \in \cl{M}$:
\[
 \xymatrix
{
F(U\tensor X) \ar[r]^-{f_{U,X}} \ar[d]_-{\eta_{U\tensor X}} & U\tensor F(X) \ar[d]^-{\id\tensor \eta_X} \\
G(U\tensor X) \ar[r]_-{g_{U,X}}& U\tensor G(X).
}
\]  

If we regard a $\cl{C}$-module structure on $\cl{M}$ as a C*-tensor functor
$\map{(\Phi,\phi)}{\cl{C}}{[\cl{M},\cl{M}]_{\mathrm{b}}}$, then 
$[\cl{M},\cl{M}]^{\cl{C}}_{\mathrm{b}}$ can be thought as something like 
a Drinfeld center. Actually, if we are given a $\cl{C}$-module functor
$(F,f)$ from $\cl{M}$ to $\cl{M}$, then $f$ defines a unitary natural
transformation $c$ from $F\circ \Phi(\tend)$ to $\Phi(\tend)\circ F$ 
satisfying the braiding equation
$c_{U \tensor V} = (\phi_{U,V}\tensor\id_F)(\id\tensor c_V)(c_U\tensor\id)(\id\tensor\phi_{U,V}*)$. The converse also holds, and then a morphism of 
$\cl{C}$-module functor can be considered as a morphism of unitary 
half-braiding.

\subsection{Index of conditional expectation}

In this subsection, we collect definitions and facts on indices of 
conditional expectations, introduced in \cite{Wa90}.

Let $B$ be a unital C*-algebra and $A$ be a unital C*-subalgebra of $B$ with
a conditional expectation $\map{E}{B}{A}$. Then we say that $E$ is 
\emph{with finite index} if it admits a quasi-basis i.e. a finite family
$(u_i)_{i = 1}^n \subset A$ satisfying 
\[
 a = \sum_{i = 1}^n u_iE(u_i^*a)
\]
for any $a\in A$. In this case we can define the \emph{index} 
of $E$ by the formula
$\Index{E} = \sum_{i = 1}^n u_iu_i^*$. This is independent of the choice of 
a quasi-basis, and $\Index{E}$ is an element of $Z(A)^{\times}$.
 On the other hand, we also have a more classical notion of index based on
the Pimsner-Popa inequality {\cite[Proposition 2.1.]{PP86}}. 
The following value is called the \emph{probabilistic index} of $E$ ({\cite[Definition 1.1.1]{Po95}}):
\[
 \Index^p E = \min\{c > 0\mid cE - \id_A\text{ is positive.}\}. 
\]
We also define the \emph{scalar index} of $E$ by replacing the positivity 
by the complete positivity:
\[
 \Index^s E = \min\{c > 0\mid cE - \id_A\text{ is completely positive.}\}. 
\]
 Obviously we have $\Index^p E \le \Index^s E \le \infty$.
Moreover the following proposition holds.

\begin{prop}[{\cite[Theorem 1.]{FK00}}] \label{prop:FK00}
Let $\map{E}{B}{A}$ be a conditional expectation. Then 
the probabilistic index of $E$ is finite if and only if the 
scalar index of $E$ is finite.
\end{prop}

We use the symbol $B_E$ for a Hilbert $A$-module obtained by 
taking the completion of $A$ with respect to a $B$-valued inner 
product $\ip{x,y}_B = E(x^*y)$. The following generalization of 
{\cite[Proposition 2.1.5.]{Wa90}} may be well-known 
for experts (c.f. {\cite[Th\'eor\`eme 3.5.]{BDH}}).
For its proof, see Subsection \ref{subsection:A2}.

\begin{prop} \label{prop:finiteness}
Let $\map{E}{B}{A}$ be a conditional expectation. 
Then $E$ is with finite index if and only if 
it satisfies both of the following two conditions:
\begin{enumerate}
 \item $\Index^p E < \infty$. 
 \item The Hilbert $A$-module $B_E$ can be decomposed into a direct sum of 
finitely generated projective Hilbert $A$-modules. 
\end{enumerate}
In this case we have $\Index^s E = \nor{\Index{E}}$.
\end{prop}

\section{Equivariant finite quantum covering spaces}
\subsection{Definition and characterizations}

Let $\G$ be a compact quantum group.
If $A \subset B$ is an inclusion of 
$\G$-C*-algebras, we say that $\map{E}{B}{A}$ is a 
\emph{$\G$-expectation} when $E$ is a $\G$-equivariant conditional 
expectation.

\begin{defn}
Let $A \subset B$ be a unital inclusion of unital $\G$-C*-algebras. 
We say that $A \subset B$ is a \emph{finite quantum $\G$-covering space}
over $A$ if it admits a $\G$-expectation $\map{E}{B}{A}$ with finite index. 
\end{defn}

In this paper, we only treat with finite quantum $\G$-covering spaces
over quantum homogeneous spaces. In such cases, we can replace 
the finiteness of the index by the finiteness of the probabilistic index. 

Let $A$ be a quantum homogeneous space of $\G$.

\begin{thrm} \label{thrm:characterization}
Let $A \subset B$ be
a unital inclusion of unital $\G$-\cstar-algebras with a 
$\G$-expectation $E$. Then the following conditions are equivalent:
\begin{enumerate}
 \item The index of $E$ is finite.
 \item The probabilistic index of $E$ is finite.
\end{enumerate}
Moreover, under these conditions, we can take a quasi-basis of $E$ in 
the algebraic core $\cl{B}$ of $B$.
\end{thrm}
\begin{proof}
At first one should note that $B_E$ can be made into a 
$\G$-equivariant Hilbert $A$-module by the action of $\G$ on $B$. 
Then $B_E$ decomposes into a direct sum of finitely generated 
projective Hilbert $A$-modules by {\cite[Theorem 6.21.]{DC16}}, 
hence the equivalence of (i) and (ii) follows from Proposition 
\ref{prop:finiteness}.

Next we show the last statement. Since $B_E$ is a finitely generated 
$\G$-equivariant Hilbert $A$-module, it has a irreducible decomposition. 
Hence we can take a finite dimensional
unitary representation $\pi$ of $\G$ and a embedding 
$\map{V}{B_E}{\cl{H}_{\pi}\otimes A}$ of $\G$-equivariant Hilbert 
$A$-module by {\cite[Theorem 6.23.]{DC16}}.
Then $V$ and its adjoint $V^*$ preserve their algebraic core, hence
we can obtain a quasi-basis $(V^*(e_i\tensor 1))_{i = 1}^n$ in $\cl{B}$
where $(e_i)_{i = 1}^n$ is an orthonormal basis of $\cl{H}_{\pi}$.
\end{proof}

\begin{coro}
Let $A \subset B$ be a finite quantum $\G$-covering space. Then 
for any intermediate unital $\G$-C*-subalgebra $C$ of $A \subset B$, 
$A \subset C$ is also a finite quantum $\G$-covering space. 
\end{coro}
\begin{proof}
Fix a $\G$-expectation $\map{E}{B}{A}$ with finite index.
Then we can easily check the finiteness of the restriction of $E$ on $C$
by using the condition (ii) of Theorem \ref{thrm:characterization}.
\end{proof}

If $\G$ is of Kac type, we can drop the equivariance of a conditional 
expectation by the averaging procedure. 
Let $A,\,B$ be $\G$-\cstar-algebras and
$\cl{A,\,B}$ be their algebraic cores respectively. 
If we are given a $\C$-linear map $\map{\phi}{\cl{B}}{A}$, 
then its \emph{equivariantization} is a $\C$-linear map
$\map{\tilde{\phi}}{\cl{B}}{A}$ given by the following:
\[
\tilde{\phi}(x) 
= (\id\otimes h)(\alpha(\phi(x_{(0)}))(1\tensor S(x_{(1)}))).
\] 
Here $h$ and $S$ denote the Haar state and the antipode of $\G$ 
respectively. By using the strong bi-invariance of $h$
({\cite[Proposition 5.24, Corollary 5.35]{KV00}}), 
we can see that $\tilde{\phi}$ is a $\G$-equivariant map from
$\cl{B}$ to $\cl{A}$. 
\begin{lemm}
Let $\map{\phi}{B}{A}$ be a c.p.~map. 
If $\G$ is of Kac type, the equivariantization of $\phi$ extends
to a $\G$-equivariant c.p.~map from $B$ to $A$.
\end{lemm}
\begin{proof}
Since $\G$ is of Kac type, we have
\[
 \tilde{\phi}(y^*x) = (\id\tensor h)((1\tensor S(y_{(1)})^*)\alpha(\phi(y_{(0)}^*x_{(0)}))(1\tensor S(x_{(1)}))).
\]
This equality implies the complete positivity of $\tilde{\phi}$,
hence we can apply Proposition \ref{prop:automatic continuity}.
\end{proof}

\begin{prop} \label{prop:deequivariantization}
Assume $\G$ is of Kac type. For a unital inclusion $A \subset B$ of 
unital $\G$-C*-algebras with $A^{\alpha} = \C 1_A$, the following 
conditions are equivalent:
\begin{enumerate}
 \item The inclusion $A \subset B$ is a finite quantum $\G$-covering space.
 \item There is an expectation $\map{E}{B}{A}$ with $\Index^p E < \infty$.
\end{enumerate}
\end{prop}
\begin{proof}
Let $E$ be a conditional expection from $B$ to $A$.
At first we have to show
that the equivariantization $\tilde{E}$ of $E$ is again
a conditional expectation. Since $\tilde{E}$ is a u.c.p.~map 
by the previous proposition, it suffices to show $\tilde{E}$
is $\cl{A}$-bimodule map on $\cl{B}$. 
Take $a \in \cl{A}$ and $x \in \cl{B}$. Then we have
\begin{align*}
 \tilde{E}(xa)
&= (\id\tensor h)(\alpha(E(x_{(0)}a_{(0)}))(1\tensor S(x_{(1)}a_{(1)}))) \\
&= (\id\tensor h)(\alpha(E(x_{(0)}))\alpha(a_{(0)})(1\tensor S(a_{(1)}))(1\tensor S(x_{(1)}))) \\
&= (\id\tensor h)(\alpha(E(x_{(0)}))(a\tensor S(x_{(1)}))) \\
&=\tilde{E}(x)a.
\end{align*}
The equality $\tilde{E}(ax) = a\tilde{E}(x)$ can be seen as follows:
\[
\tilde{E}(ax) = \tilde{E}(x^*a^*)^* = (\tilde{E}(x)^*a^*)^* = a\tilde{E}(x).
\]
Now the probabilistic index of $E$ is finite, hence the scalar index $c$
of $E$ is also finite by Proposition \ref{prop:FK00}. Then 
the equivariantization of $cE - \id_B$ is a c.p.~map and 
coincides with $c\tilde{E} - \id_B$. This implies that $\tilde{E}$ has the 
finite scalar index, hence $A\subset B$ is a finite quantum $\G$-covering
space.
\end{proof} 

\subsection{Covering degree}

Next we consider a notion which is analogous to the covering degree 
of covering space.

\begin{defn}
Let $A \subset B$ be a finite quantum $\G$-covering space. Then the 
covering degree of $A \subset B$ is the infimum of 
the scalar indices over all $\G$-expectations from $B$ to $A$.
\end{defn}
\begin{rema}
If $\G$ is of Kac type, the covering degree coincides with the infimum of
the scalar indices of all conditional expectations of the inclusion.
This fact can be seen from the proof of Proposition 
\ref{prop:deequivariantization}.
\end{rema}

To calculate the covering degree of a given finite quantum $\G$-covering 
space, it is usuful to regard it as the dimension of a rigid object of 
a suitable \cstar-tensor category.

Let $A$ and $B$ be quantum homogeneous spaces of $\G$.
A \emph{$\G$-equivariant $(A,B)$-correspondence} 
is a pair of $\G$-equivariant Hilbert $B$-module $E$ and 
a unital $\G$-equivariant \star-homomorphism from $A$ to $\cl{L}_B( E)$. 
The category of $\G$-equivariant $(A,B)$-correspondences is denoted 
by $\G\text{-}\corr_{A,B}$, and its full subcategory consisting of right-finitely generated correspondences is denoted by 
$\G\text{-}\corr^{\rfin}_{A,B}$. 
We also use the symbols $\G\text{-}\corr_A$ and 
$\G\text{-}\corr_A^{\fin}$ when $A = B$.

Let $A \subset B$ be an inclusion of unital $\G$-\cstar-algebras with 
$A^{\alpha} = \C 1_A$. If we are given a $\G$-expectation $\map{E}{B}{A}$
with finite index, then we have a $\G$-correspondence $B_E$ over $A$, 
on which the left action is given by the left multiplication.
Moreover $B_E$ can be made into a \cstar-Frobenius algebra ({\cite[Section 3.1]{BKLR}}) in 
$\G\text{-}\corr_A^{\fin}$ with the following maps:
\begin{itemize}
 \item $\map{m}{B_E\tensor_A B_E}{B_E}$, induced by the multiplication of 
$m$, 
 \item $\map{\iota}{A}{B_E}$, induced by the inclusion $A \subset B$.
\end{itemize}
By using a quasi-basis $(v_i)_{i = 1}^n$ of $E$, 
$\map{m^*}{B_E}{B_E\tensor_A B_E}$ can be calculated as follows:
\[
 m^*(x) = \sum_{i = 1}^n v_i\tensor v_i^*x
\]
Hence we have $mm^*(1_B) = \Index E$ and 
$\iota^* mm^*\iota(a) = E(\Index E) a$.
Since $\iota^*(b) = E(b)$ for any $b \in B$, we also have
\[
 d(B_E) \le \nor{\iota^* mm^*\iota} = E(\Index E) \le \nor{\Index E},
\]
here the left-hand side is the dimension of $B_E$ as an object of 
$\G\text{-}\corr^{\fin}_A$. Moreover, if $(B_E,m,\iota)$ is a Q-system,
these inequalities turns out to be equalities. 
Conversely any \cstar-Frobenius algebra in $\G\text{-}\corr^{\fin}_A$
gives a pair of a finite quantum $\G$-covering space $A \subset B$ and
$\G$-expectation $\map{E}{B}{A}$ with finite index. 
Hence we can get the following proposition since 
any Frobenius \cstar-algebra is isomorphic to 
a Q-system, as shown in {\cite[Theorem 2.9]{NY18}}
\begin{prop} \label{prop:fundamental}
 Let $A$ be a quantum homogeneous space of $\G$.
\begin{enumerate}
 \item There is a one-to-one correspondence between finite quantum $\G$-covering spaces of $A$ and Q-systems in $\G\text{-}\corr^{\fin}_A$.
 \item Let $A \subset B$ be a finite quantum $\G$-covering space. 
Its covering degree $d$ coincides with the dimension of the corresponding 
Q-system. Moreover, there is a $\G$-expectation $\map{E}{B}{A}$ with
$\Index E = d1_B$. 
\end{enumerate}
\end{prop}

\begin{exam} \label{exam:adjoint action}
When $A = \C$, the trivial $\G$-\cstar-algebra, then 
$\G\text{-}\corr^{\fin}_{\C}$ is nothing but $\rep^{\fin}\G$.
Hence the covering degree of a finite quantum $\G$-space over $\C$
coincides with the quantum dimension as a unitary representation. In
particular, the covering degree of $\C \subset B(\cl{H}_{\pi})$ is 
$(\dim_q \pi)^2$ if we consider the adjoint action on $B(\cl{H}_{\pi})$
\end{exam}

As an application of Proposition \ref{prop:fundamental}, we show
that the covering degree of a finite quantum $\G$-spaces over 
$C(\G)$ must be a positive integer. 

Let $\rep \cl{O}(\G)$ be the category of unital \star-representations 
of $\cl{O}(\G)$. By using the coproduct on $\cl{O}(\G)$, 
we can make $\rep \cl{O}(\G)$ into a \cstar-tensor category.
For any $(\pi,\cl{H}_{\pi}) \in \rep\cl{O}(\G)$, we can construct a
$\G$-equivariant correspondence $E_{\pi}$ over $C(\G)$ as follows:
\begin{itemize}
 \item As a Hilbert $C(\G)$-module, $E_{\pi} = \cl{H}_{\pi}\tensor C(\G)$.
 \item The left action of $\G$ is given by $\id\tensor \Delta$.
 \item The left action of $C(\G)$ is given by $(\pi\tensor \lambda)\Delta$,
where $\lambda$ is the left multiplication of $C(\G)$ on $C(\G)$.
\end{itemize}
Here $(\pi\tensor \lambda)\Delta$ can be defined on $C(\G)$
by Lemma \ref{lemm:absorption}.

The following proposition is a quantum analogue of 
{\cite[Proposition 2.2]{AV16}}.
\begin{prop} \label{prop:corr over C(G)}
The functor $(\pi,\cl{H}_{\pi}) \longmapsto E_{\pi}$ gives
an equivalence of \cstar-tensor categories from $\rep\cl{O}(\G)$
to $\G\text{-}\corr_{C(\G)}$.
\end{prop}

Before proving this, we should prepare the following lemma.

\begin{lemm}
Let $E$ be a $\G$-equivariant Hilbert $C(\G)$-module.
The the following formula gives an semi-inner product on the
algebraic core $\cl{E}$ of $E$:
\[
 \ip{\xi,\eta} = \eps(\ip{\xi,\eta}_{C(\G)})
\]
\end{lemm}
\begin{proof}
The only non-trivial part is the positivity of $\ip{\xi,\xi}$ for each
$\xi \in \cl{E}$. By replacing $E$ with its $\G$-equivariant Hilbert $C(\G)$
submodule generated by $\xi$ and $\eta$, we may assume $E$ is 
finitely generated. 

If we are given another 
$F \in \G\text{-}\mod^{\fin}_{C(\G)}$ and an isometry 
$V \in \cl{L}^{\G}_{C(\G)}(E,F)$, we have
\[
 \ip{V\xi,V\eta} = \eps(\ip{V\xi,V\eta}_{C(\G)}) = \eps(\ip{\xi,\eta}_{C(\G)}) = \ip{\xi,\eta}.
\]
Since $\G\text{-}\mod^{\fin}_{C(\G)}$ is connected as a 
$\rep^{\fin}\G$-module category, the equality above implies that
it suffices to show the statement for $\cl{H}_{\pi}\tensor C(\G)$ 
with an arbitrary $\pi \in \rep^{\fin}\G$ 
 But this is trivial since we have
\[
 \ip{v\tensor x,u\tensor y} = \ip{v,u}\eps(x^*y).
\]
\end{proof}

\begin{proof}[Proof of Proposition \ref{prop:corr over C(G)}]
At first we make the functor in the statement into a \cstar-tensor functor.
This can be completed by associating the following unitary:
\begin{align*}
\xymatrix@!C@R=6pt
{
 (\cl{H}_{\pi}\tensor C(\G))\tensor_{C(\G)} (\cl{H}_{\rho}\tensor C(\G)) 
\ar[r]^-{U_{\pi,\rho}}& (\cl{H}_{\pi}\tensor\cl{H}_{\rho})\tensor C(\G). \\
 (\xi\tensor x)\tensor (\eta \tensor y) \ar@{|->}[r]& (\xi\tensor \rho(x_{(1)})\eta) \tensor x_{(2)}y.
}
\end{align*}

It can be easily seen that this unitary satisfies the conditions in the
definition of \cstar-tensor functor.

To show that this functor gives the equivalence, we construct
a quasi-inverse functor. Let $E$ be a $\G$-equivariant corrrespondence 
over $C(\G)$. Set $\cl{H}_E$ as a completion of 
the algebraic core $\cl{E}$ of $E$,
with respect the inner product in the previous lemma.
Then each element $x \in \cl{O}(\G)$ acts on $\cl{H}_E$ from the 
left, since $\cl{O}(\G)$ is a linear span of matrix coefficients of 
$\cl{O}(\G)$-valued unitary matrices. Now we have a \star-representation
$(\pi_E,\cl{H}_{\pi})$ of $\cl{O}(\G)$.

To complete the proof, we will check that these functors give a
quasi-inverse of each other.
For $(\pi,\cl{H}_{\pi}) \in \rep^{\fin}\cl{O}(\G)$, the algebraic core
of $E_{\pi}$ is precisely $\cl{H}_{\pi}\tensor \cl{O}(\G)$. Hence 
we have a unitary from $\cl{H}_{E_{\pi}}$ to $\cl{H}_{\pi}$ given by
$\xi\tensor x \longmapsto \eps(x)\xi$ and this is an intertwiner
of representations of $\cl{O}(\G)$. On the other hand, for any 
$E \in \G\text{-}\corr^{\rfin}_{C(\G)}$, 
we have a map from $\cl{E}$ to $\cl{H}_E\tensor C(\G)$  given as follows:
\[
 \xi\in \cl{E} \longmapsto \xi_{(0)}\tensor \xi_{(1)} \in \cl{H}_E\tensor C(\G).
\]
This map is a $\G$-equivariant isometry and 
intertwines with left $C(\G)$-action. To show that this is a right 
$C(\G)$-module map, one should note the following: 
For any $x \in \cl{O}(\G)$ and $\xi \in \cl{E}$, we have
\[
 \nor{\eps(x)\xi - \xi x}_{\cl{H}_E}^2 
= \eps(\ip{\eps(x)\xi - \xi x,\eps(x)\xi - \xi x}_{C(\G)}) 
= 0.
\]
Hence we have $\xi x = \eps(x)\xi$ in $\cl{H}_E$ and
\[
 (\xi x)_{(0)}\tensor (\xi x)_{(1)} 
= \eps(x_{(0)})\xi_{(0)}\tensor \xi_{(1)}x_{(1)}
= \xi_{(0)}\tensor \xi_{(1)}x.
\]
This implies that the map above is an embedding of $\G$-equivariant 
correspondence from $E$ to $\cl{H}_E\tensor C(G)$.
Hence the range of this map contains $\xi_{(0)}\tensor \xi_{(1)}x$ 
for any $\xi \in \cl{E}$ and $x \in C(\G)$. Combining this with 
$\cspan \alpha_E(E)(\C\tensor C(\G)) = E\tensor C(\G)$, 
we can see that this map is unitary. Now we complete the proof.
\end{proof}
\begin{rema}
We also have an equivalence 
$\rep^{\fin}\cl{O}(\G) \cong \G\text{-}\corr^{\rfin}_{C(\G)}$.
\end{rema}

Let us recall the quantum Bohr compactification due to So\l tan \cite{So05}.
For the dual discrete quantum group $\hat{\G}$ of $\G$,
its quantum Bohr compactification $\mathfrak{b}\hat{\G}$ is of Kac type
({\cite[Theorem 4.5]{So05}}). Moreover every finite dimensional
\star-representation of $\cl{O}(\G)$ gives rise to a finite dimensional
unitary representation of $\mathfrak{b}\hat{\G}$ 
(c.f. {\cite[Theorem 2.3]{CKS}}). These facts imply that
$\rep^{\fin}\cl{O}(\G)$ is rigid and the forgetful functor 
$\rep^{\fin}\cl{O}(\G)\longmapsto \hilb^{\fin}$ is dimension-preserving.
Now Proposition \ref{prop:corr over C(G)} yields the desired result.

\begin{coro} \label{coro:integer degree}
For a finite quantum $\G$-covering space over $C(\G)$,
its covering degree must be a positive integer. 
\end{coro}

% We end this subsection with a relation between the covering degree and
% the probabilistic index.

% \begin{prop}
% Assume $\G$ is coamenable. Let $A \subset B$ be a finite quantum 
% $\G$-covering space. If $B$ is a right coideal of $\G$, 
% it admits a unique $\G$-expectation $\map{E}{B}{A}$
% and its covering degree coincides with the probabilistic index of $E$.
% \end{prop}
% \begin{proof}
% The uniqueness follows from $B^{\G} = \C1_B$. Hence 
% it suffices to show that the probablisitic index of $E$ 
% coincides with the scalar index of $E$.

% Since $\G$ is coamenable, the counit $\eps$ of $\G$ is defines on
% $C(\G)$. In particular we can define a state $\phi$ on $B$ by 
% $\phi(b) = \eps(E(b))$. Then this satisfies 
% $E(b) = (\phi\tensor \id)\Delta(b)$
% since
% \[
%  (\phi\tensor\id)\Delta(b) = (\eps\tensor\id)(E\tensor\id)\Delta(b)
% =(\eps\tensor \id)\Delta(E(b)) = E(b).
% \]
% Now let $c > 0$ be the probabilistic index of $E$. Then we have 
% $c\phi - \eps  = \eps\circ (cE - \id_B) \in B^*_+$. Hence
% $(c\phi - \eps)\tensor\id$ is a c.p.~map from $B\tensor C(\G)$ to $C(\G)$,
% in particular
% \[
%  cE - \id = ((c\phi - \eps)\tensor \id)\circ \Delta
% \]
% is also a c.p.~map. This implies the scalar index of $E$ can be bounded
% from above by the probabilistic index of $E$. On the other hand
% we always have the converse inequality by the definitions of them. 
% This completes the proof.
% \end{proof}

\subsection{A module categorical description}

Let us recall the Tannaka-Krein type result for quantum 
homogeneous spaces {\cite[Theorem 6.4]{DY13}}, {\cite[Theorem 3.3]{Ne14}}. 
This suggests us the possibility to 
understand finite quantum $\G$-covering space by the language of 
module categories. 

Let $M$ be a right-finitely generated $\G$-equivariant 
$(A,B)$-correspondence.
Then we have a $\rep^{\fin}\G$-module functor from $\G\text{-}\mod_A^{\fin}$
to $\G\text{-}\mod_B^{\fin}$ given by the following:
\begin{itemize}
 \item As a \cstar-functor, it is given by taking the internal tensor 
product with $E$ over $A$ i.e. the functor sends $E$ to $E \tensor_A M$. 
Here $E\tensor_A M$ is considered as a $\G$-equivariant Hilbert 
$B$-module by $\xi\tensor m \longmapsto \alpha_E(\xi)_{13}\beta_M(m)_{23}$.
Since $E$ and $M$ is finitely generated, this Hilbert $B$-module is also
finitely generated.
 \item The unitary from $(\cl{H}_{\pi}\tensor E)\tensor_A M$ to 
$\cl{H}_{\pi}\tensor (E\tensor_A M)$ is given by 
$(v \tensor \xi) \tensor m \longmapsto v \tensor (\xi \tensor m)$.
\end{itemize}
This $\rep^{\fin}\G$-module functor is denoted by $\tend\tensor_A M$.

The following is a main theorem of this subsection, which is 
a generalization of {\cite[Theorem 7.1]{DY13}}
\begin{thrm} \label{thrm:fundamental theorem}
Let $A$ and $B$ be quantum homogeneous spaces of $\G$. Then 
the \cstar-functor $M \longmapsto \tend\tensor_A M$ gives 
the following equivalence of \cstar-categories:
\[
 \G\text{-}\corr_{A,B}^{\rfin} \cong 
[\G\text{-}\mod_A^{\fin},\G\text{-}\mod_B^{\fin}]_{\mathrm{b}}^{\rep^{\fin}\G}
\]
\end{thrm}

We would like to use an argument used in {\cite[Theorem 7.1]{DY13}} 
for the proof of the above theorem. 

At first, we fix a complete set $\irr{\G}$ of representatives of 
all equivalence classes of irreducible unitary representations of $\G$. 
For any $\pi \in \irr{\G}$, $\bar{\pi}$ denotes an element of 
$\irr{\G}$ unitary equivalent to the conjugate representation of $\pi$,
which is uniquely determined.

Let $E$ be a $\G$-equivariant Hilbert $A$-module and $\cl{E}$ be its 
algebraic core. For each $\pi \in \irr\G$, $\cl{E}_{\pi}$ denotes 
the space $\cl{L}_A^{\G}(\cl{H}_{\pi}\tensor A,E)\tensor \cl{H}_{\pi}$.
Then we can regard $\cl{E}_{\pi}$ as a subspace of $\cl{E}$ via 
$T\tensor \xi \longmapsto T(\xi\tensor 1)$ and have the spectral 
decomposition $\cl{E} = \bigoplus_{\pi} \cl{E}_{\pi}$.
By considering $A$ as a $\G$-equivariant Hilbert $A$-module, this 
decomposition also can be applied to $\cl{A}$. 

Now we can present the right $\cl{A}$-action, the $\cl{A}$-valued
inner product and the $\G$-action on $\cl{E}$ as follows: 
For $T\tensor \xi \in \cl{E}_{\pi},\,S\tensor \eta \in \cl{E}_{\rho}$ and
$s\tensor v \in \cl{A}_{\rho}$, we have
\begin{align*}
 (T\tensor \xi)(s\tensor v) 
&= \sum_{\sigma \in \irr{\G}} \sum_{V \in (\sigma,\pi\rho)}  
T(\id_{\pi}\tensor s)(V\tensor \id_A)\tensor V^*(\xi\tensor v),\\
 \ip{T\tensor\xi,S\tensor\eta}_{\cl{A}}
&= \sum_{\sigma \in \irr{\G}} \sum_{V \in (\sigma,\bar{\pi}\rho)}
(R_{\pi}^*\tensor\id_A)(\id_{\bar{\pi}}\tensor T^*S)(V\tensor \id_A)\\[-1.3em]
&\hspace{0.4\textwidth}\tensor(\xi^*\tensor V^*)(\bar{R}_{\pi}(1)\tensor \eta),\\
 \alpha_E(T\tensor \xi) &= T \tensor U_{\pi}(\xi\tensor 1).
\end{align*}
Here we use the following notations:
\begin{itemize}
 \item The quadruple $(\pi,\bar{\pi},R_{\pi},\bar{R}_{\pi})$ denotes a 
standard solution in $\rep^{\fin}\G$. 
 \item The symbol $\sum_{V\in (\sigma,\pi\rho)}$ means the summation over
all elements of a fixed orthonormal basis of 
$\mathrm{Hom}_{\G}(\sigma,\pi\tensor\rho)$, equipped with an inner product 
$\ip{V,W} = V^*W$. 
\end{itemize}
These formulae do not depend on the choices of the standard solutions and 
the orthonormal bases. We also have the following presentation of 
the involution of $\cl{A}$ under these notations:
\[
 (s\tensor v)^*
= (R_{\pi}^*\tensor\id)(\id_{\bar{\pi}}\tensor s^*)
\tensor(v^* \tensor\id)\bar{R_{\pi}}(1). 
\]
Now we finishes the preparation for the proof of Theorem 
\ref{thrm:fundamental theorem}.

\begin{proof}[Proof of Theorem \ref{thrm:fundamental theorem}]
At first we construct a functor from 
$[\G\text{-}\mod_A^{\fin},\G\text{-}\mod_B^{\fin}]
_{\mathrm{b}}^{\rep^{\fin}\G}$ to $\G\text{-}\corr_{A,B}^{\rfin}$.
Let $(F,f)$ be a $\rep^{\fin}\G$-module functor from 
$\G\text{-}\mod_A^{\fin}$ to $\G\text{-}\mod_B^{\fin}$. Then we can 
associate a left action of $\cl{A}$ on $F(A)$ as follows: For
$s \tensor v \in \cl{A}_{\pi}$ and $\eta \in F(A)$, the left multiplication
by $s\tensor v$ is given by
\[
 (s\tensor v)\eta = F(s)f_{\pi,A}^*(v\tensor \eta).
\]
Then some direct calculations show that this defines an adjointable
$B$-module map. Moreover we can see the following relations:
\begin{align*}
 \ip{\eta,(s\tensor v)\eta'}_{B}
&= \ip{(s\tensor v)^*\eta,\eta'}_{B},\\
 \beta_{F(A)}((s\tensor v)\eta) 
&= \alpha(s\tensor v)\beta_{F(A)}(\eta).
\end{align*}
Hence Proposition \ref{prop:automatic continuity} shows that this 
left $\cl{A}$-action extends to
a left $A$-action on $F(A)$ compatible with the $\G$-actions.
Now $F(A)$ is equipped with a structure of $\G$-equivariant 
$(A,B)$-correspondence. 
Moreover if we are given another 
$\rep^{\fin}\G$-module functor $(G,g)$ and a morphism 
$\map{\theta}{(F,f)}{(G,g)}$, its component 
$\map{\theta_A}{F(A)}{G(A)}$ is a morphism
of $\G$-equivariant $(A,B)$-correspondence. These constructions
give rise to a functor we want.

We will complete the proof by showing that 
these functors give a quasi-inverse of each other.
Fix an right-finitely generated $\G$-equivariant $(A,B)$-correspondence 
$M$. Then $A\tensor_A M$ is isomorphic to $M$ as a $\G$-equivariant 
Hilbert $B$-module map via $x\tensor \xi \longmapsto x\xi$. Hence it 
suffices to show the compatibility with the left $A$-actions.
For $s\tensor v \in \cl{A}_{\pi}$ and $a\tensor \xi \in A\tensor_A M$, 
the action as above is calculated as follows:
\[
 (s\tensor v)(a\tensor \xi) = (s\tensor \id_M)((v\tensor a)\tensor\xi)
=s(v\tensor a)\tensor \xi.
\]
Since $s\tensor v$ corresponds to $s(v\tensor 1) \in \cl{A}$,
this equality shows the compatibility we required.

Next we take a $\rep^{\fin}\G$-module functor $(F,f)$.  
For any finitely generated $\G$-equivariant Hilbert $A$-module $E$, 
we have a map from $\cl{E}\tensor F(A)$ to $F(E)$ which sends
$(T\tensor v)\tensor \xi$ to $F(T)f_{\pi,A}^*(v\tensor \xi)$.
Then we have
\begin{align*}
& \ip{F(T)f_{\pi,A}^*(v\tensor\xi),F(S)f_{\rho,A}^*(w\tensor \eta)}_B\\
&=\ip{v\tensor\xi,f_{\pi,A}F(T^*S)f_{\rho,A}^*(w\tensor \eta)}_B \\
&=\ip{v\tensor\xi,(\id_{\pi}\tensor (R_{\pi}^*\tensor\id_{F(A)})(\id_{\bar{\pi}}\tensor f_{\pi,A}F(T^*S)f_{\rho,A}^*))(\bar{R}_{\pi}\tensor\id_{\rho}\tensor\id_{F(A)})(w\tensor\eta)}_B
\end{align*}
By using 
$(R_{\pi}^*\tensor\id_{F(A)})f_{\pi\tensor\bar{\pi},A} = F(R_{\pi}^*\tensor\id_A)$
and $f_{\pi\tensor\bar{\pi},A} = (\id_{\bar{\pi}}\tensor f_{\pi,A})f_{\bar{\pi},\pi\tensor A}$, we also have
\begin{align*}
(R_{\pi}^*\tensor\id_{F(A)})(\id_{\bar{\pi}}\tensor f_{\pi,A}F(T^*S)f_{\rho,A}^*)
=F((R_{\pi}^*\tensor\id_A)(\id_{\bar{\pi}}\tensor T^*S))f_{\bar{\pi}\tensor\rho,A}^*.
\end{align*}
Hence the map induces an isometry from $E\tensor_A F(A)$ to $F(E)$. It 
can be easily seen that this isometry is a $\G$-equivariant Hilbert 
$B$-module map. To prove the unitarity, one should note that $F$ preserves
the range projection i.e. for any 
$T \in \cl{L}_A^{\G}(\cl{H}_{\pi}\tensor A,E)$ with the range projection
$p \in \cl{L}_A^{\G}(E)$, the range projection of 
$\map{F(T)}{F(\cl{H}_{\pi}\tensor A)}{F(E)}$ is $F(p)$. 
This fact implies that 
$F(E)$ is contained the sum space of the range of $F(T)$ over all 
$T \in \bigcup_{\pi\in\irr{\G}}\cl{L}_A^{\G}(\cl{H}_{\pi}\tensor A,E)$.  
Hence the map from $E\tensor_A F(A)$ to $F(E)$ is unitary.
Now the proof is completed.
\end{proof}

As an application of this theorem, we can see that
each Jones' value  appears as the covering degree of a finite quantum 
$\G$-covering space for some $\G$.

\begin{thrm} \label{thrm:Jones value}
 The covering degree of a finite quantum $\G$-covering space is
contained in $\{4\cos^2(\pi/n)\mid n \ge 3\}\cup[4,\infty)$. 
Conversely, for any element $d$ of this set, we have a
compact quantum group $\G$ and a finite quantum $\G$-covering space
$A\subset B$ with its covering degree $d$.
\end{thrm}

In the following proof, we use the monoidal opposite $\cl{C}^{\mathrm{op}}$
of a given \cstar-tensor category $\cl{C}$ 
(c.f. {\cite[Definition 2.1.5]{EGNO}}). As a \cstar-category,
$\cl{C}^{\mathrm{op}} = \cl{C}$. But its tensor product 
$\map{\tend\tensor\tend}{\cl{C}^{\mathrm{op}}\times\cl{C}^{\mathrm{op}}}{\cl{C}^{\mathrm{op}}}$ 
is given by $(X,Y)\longmapsto Y\tensor X$.

\begin{proof}
The first half of the statement follows from a general theory on
Q-system: Take a finite quantum $\G$-covering space $A \subset B$ and
let $d$ be its covering degree.
Then we have a Q-system $X$ in $\G\text{-}\corr^{\fin}_A$ correspoinding
to $A \subset B$ with $d(X) = d$ (Proposition \ref{prop:fundamental}). Let $\cl{C}$ be a \cstar-tensor category generated by $X$. 
Then we can find a finite factor $N$ and a fully faithful embedding
$\map{i}{\cl{C}}{\mathrm{Bimod}^{\fin}_N}$ by {\cite[Theorem 3.9]{Ya03}}.
Since a Q-system in $\mathrm{Bimod}^{\fin}_N$ corresponds to
a finite extension of von Neumann algebra {\cite[Theorem 6.1]{Lo94}},
we have a finite extension $N \subset M$ whose minimum index is 
$d(i(X)) = d(X) = d$.
Now our assertion follows from the Jones' result on the value of index
({\cite[Theorem 4.3.1]{Jo83}}). 

To prove the second half, take an arbitrary $d$ from the set
in the statement. If $d \ge 4$, we have a $q \in (0,1)$ such that
the fundamental representaion $\pi_{1/2}$ of $SU_q(2)$ satisfies
$\dim_q \pi_{1/2} = \sqrt{d}$. Then Example \ref{exam:adjoint action}
gives a finite quantum $SU_q(2)$-covering space with the covering 
degree $d$. 

Now consider the case $d < 4$. Then we can take a fusion category 
$\cl{C}$ and its object $X$ with $d(X) = \sqrt{d}$ (For example, let
$\cl{C}$ be the Temperley-Lieb-Jones category {\cite[Section 2.5]{NT13}}). 
We may assume $\cl{C}$ is generated by $X$ and its conjugate object
by replacing $\cl{C}$ with its full subcateogry.

Take a $q \in [-1,1]\setminus\{0\}$ and a 
\cstar-tensor functor $\map{\Phi}{\rep^{\fin} SU_q(2)}{\cl{C}}$ with
$\Phi(\pi_{1/2}) = X\oplus \bar{X}$. The existence of those is guaranteed 
by the universal property of $\rep^{\fin} SU_q(2)$ 
({\cite[Theorem 2.5.3]{NT13}}).
Then $\cl{C}$ can be thought as a $\rep^{\fin}SU_q(2)$-module category
by $\pi\tensor Y = \Phi(\pi)\tensor Y$ for $\pi \in \rep^{\fin}SU_q(2)$ and
$Y \in \cl{C}$.
Moreover our assumption implies that this is connected. Hence we can
use the duality
theorem {\cite[Theorem 6.4]{DY13}} to find
a quantum homogeneous space $A$ of $SU_q(2)$ for which 
there is an equivalence $\cl{C}\cong SU_q(2)\text{-}\mod_A^{\fin}$ of
$\rep^{\fin} SU_q(2)$-module categories.
On the other hand, by using the right multiplication of $\cl{C}$ 
on $\cl{C}$, we have a \cstar-tensor functor from $\cl{C}^{\mathrm{op}}$ to 
$[SU_q(2)\text{-}\mod_A^{\fin},SU_q(2)\text{-}\mod_A^{\fin}]_b^{\rep^{\fin} SU_q(2)}$.
Since any \cstar-tensor functor from a fusion category is dimension-
preserving, this implies we have an object $Y$ in 
$[SU_q(2)\text{-}\mod_A^{\fin},SU_q(2)\text{-}\mod_A^{\fin}]_b^{\rep^{\fin} SU_q(2)}$ with the dimension $\sqrt{d}$. 
Hence we have a Q-system $\bar{Y}\tensor Y \in [SU_q(2)\text{-}\mod_A^{\fin},SU_q(2)\text{-}\mod_A^{\fin}]_b^{\rep^{\fin} SU_q(2)}$, and 
we have a finite quantum $SU_q(2)$-covering space with the
 covering degree $d$ by Proposition \ref{prop:fundamental}. 
\end{proof}

\section{Induction from actions of the maximal Kac quantum subgroup}
\subsection{Module categorical interpretation of induced actions}

Let $\G$ be a compact quantum group and $\H$ be its quantum subgroup. 
For an $\H$-\cstar algebra $A$, we use the symbol
$\tilde{A}$ to present its induced $\G$-\cstar-algebra. 

Take $\pi \in \rep^{\fin}\G$ and $E \in \H\text{-}\mod^{\fin}_A$ arbitrarily.
Then the unitary $U_{\pi,13}^{-1}$ on $\cl{H}_{\pi}\tensor E\tensor C(\G)$
is restricted to an unitary operator from 
$\ind_{\H}^{\G}(\cl{H}_{\pi|_{\H}}\tensor E)$ 
to $\cl{H}_{\pi}\tensor \ind_{\H}^{\G} E$.
If we regard $\H\text{-}\mod^{\fin}_A$ as a $\rep^{\fin}\G$-module
category via the restriction functor 
$\rep^{\fin}\G\longrightarrow \rep^{\fin}\H$, the collection of 
such unitaries makes the induction functor 
$\map{\ind_{\H}^{\G}}{\H\text{-}\mod^{\fin}_A}{\G\text{-}\mod^{\fin}_{\tilde{A}}}$ 
into a $\rep^{\fin}\G$-module functor. The following theorem is 
essentially proved in {\cite[Theorem 7.3.]{Va05}} 
 
\begin{thrm} \label{thrm:module categorical induction}
The functor $\ind_{\H}^{\G}$ gives an equivalence 
$\H\text{-}\mod_A \cong \G\text{-}\mod_{\tilde{A}}$ 
as $\rep^{\fin}\G$-module categories. This equivalence also
gives an equivalence 
$\H\text{-}\mod^{\fin}_A \cong \G\text{-}\mod^{\fin}_{\tilde{A}}$.
\end{thrm}

Instead of relying on the Vaes' generalization of Green imprimitivity,
we will give a quasi-inverse of $\ind_{\H}^{\G}$.

The algebraic core of $\tilde{A}$ is denoted by $\tilde{\cl{A}}$.
The following is the explict presentation of $\tilde{\cl{A}}$:
\[
 \tilde{\cl{A}} 
= \{x \in \cl{A}\tensor \cl{O}(\G)\mid (\alpha\tensor\id)(x) = (\id\tensor \ell_{\H})(x)\}.
\]
Here $q$ is the canonical surjection from $\cl{O}(\G)$ to $\cl{O}(\H)$.

\begin{lemm}
Let $\tilde{E}$ be a $\G$-equivariant Hilbert $\tilde{A}$-module and
$\cl{\tilde{E}}$ be its algebraic core. Then the following gives an
$A$-valued semi-inner product on $\cl{\tilde{E}}$:
\[
 \ip{\xi,\eta}_{A} = (\id\tensor \eps)(\ip{\xi,\eta}_{\tilde{A}})
\]
\end{lemm}
\begin{proof}
By replacing $\tilde{E}$ with its $\G$-equivariant Hilbert 
$\tilde{A}$-submodule generated by $\xi$ and $\eta$, we may assume 
that $\tilde{E}$ is countably generated.
Then we can find a unitary representation $\pi$ of $\G$ such that 
there is an isometry 
$V \in \cl{L}_{\tilde{A}}^{\G}(\tilde{E},\cl{H}_{\pi}\tensor \tilde{A})$, 
by using the equivariant Kasparov stabilization theorem ({\cite[Th\'eor\`eme 3.2]{Ve02}}). 
Hence it suffices to show the statement for 
$\cl{H}_{\pi}\tensor \tilde{A}$. In this case we have 
$\cl{E} = \cl{H}_{\pi}\tensor \tilde{\cl{A}}$ and
\[
 \ip{v\tensor x, w\tensor y}_A 
= \ip{v,w}(\id\tensor \eps)(x)^*(\id\tensor y).
\]
for $v,w \in \cl{H}_{\pi}$ and $x,y \in \tilde{\cl{A}}$.
Now the statement follows from this.
\end{proof}

Let $\tilde{E}_0$ be a Hilbert $A$-module obtained as the completion 
of $\tilde{E}$ with respect to this $A$-valued semi-inner product. 
Then the map $\map{(\id\tensor q)\tilde{\alpha}_{\tilde{E}}}{\tilde{\cl{E}}}{\tilde{\cl{E}}\tensor \cl{O}(\H)}$ induces
a $\H$-action on $\tilde{E}_0$, making $\tilde{E}$ into an 
$\H$-equivariant Hilbert $A$-module. 

\begin{lemm} \label{lemm:expectation onto ind}
Let $E$ be an $\H$-equivariant Hilbert $A$-module. Then the following map
defined $P_E$ on $E\tensor \cl{O}(\G)$ is a surjection onto 
$\tilde{\cl{E}}$, the algebraic core of $\ind_{\H}^{\G} E$:
\[
 P_E(\xi\tensor x) 
= (\id\tensor h_{\H})((\id\tensor S)\alpha_E(\xi)(1\tensor q((x_{(1)}))))\tensor x_{(2)}.
\]
Here $h_{\H}$ is the Haar state of $\H$.
\end{lemm}
\begin{proof}
Since we have 
\[
 \tilde{\cl{E}} = \{\xi \in \cl{E}\tensor \cl{O}(\G)\mid (\alpha_E\tensor\id)(\xi) = (\id\tensor \ell_{\H})(\xi)\},
\]
it can be easily seen that $P_E$ is identical on $\tilde{E}$. The remaining
part is to show $P_E(E\tensor \cl{O}(\G)) \subset \tilde{E}$. Take
arbitrary $\xi \in E$ and $x \in \cl{O}(\G)$. Then we have the following:
\begin{align*}
 (\alpha_E\tensor \id)(P_E(\xi\tensor x))
&=(\id\tensor \id\tensor h_{\H})((\id\tensor (\id\tensor S)\Delta)(\alpha_E(\xi))(1\tensor 1\tensor q(x_{(1)})))\tensor x_{(2)}\\
&=(\id\tensor h_{\H}\tensor S^{-1})((\id\tensor \Delta\circ S)(\alpha_E(\xi))(1\tensor q(x_{(1)})\tensor 1))\tensor x_{(2)}\\
 (\id\tensor \ell_{\H})(P_E(\xi\tensor x))
&=(\id\tensor h_{\H}\tensor \id)(((\id\tensor S)(\alpha_{E}(\xi))\tensor 1)(1\tensor \Delta(q(x_{1})))) \tensor x_{(2)}.
\end{align*}
Hence the statement follows from the strong bi-invariance of $h_{\H}$.
\end{proof}
\begin{rema}
 If $\H$ is of Kac type, $P_E$ extends to a $\G$-equivariant c.c. map
from $E\tensor C(\G)$ to $\ind_{\H}^{\G} \G$. This follows from
Proposition \ref{prop:automatic continuity} by using the corner
trick to $E \subset \cl{K}_A(A\oplus E)$ and 
$\tilde{\cl{E}} \subset \cl{K}_{\tilde{A}}(\tilde{A}\oplus\tilde{E}) \cong \ind_{\H}^{\G}\cl{K}_A(A\oplus E)$.
\end{rema}

\begin{proof}[Proof of Theorem \ref{thrm:module categorical induction}]
We have two functors: The first one is $\ind_{\H}^{\G}$, and 
the second one is 
$\map{\mathrm{Ev}_{\G}^{\H}}{\G\text{-}\mod_{\tilde{A}}}{\H\text{-}\mod_A}$
given by $\tilde{E} \longmapsto \tilde{E}_0$. The remaining 
is to show that these functors are quasi-inverses of each other.

Take $E \in \H\text{-}\mod_A$ arbitrarily. Then we have a linear map
$\map{(\id\tensor \eps)|_{\tilde{\cl{E}}}}{\tilde{\cl{E}}}{\cl{E}}$,
for which we have
\[
 \ip{(\id\tensor \eps)(\xi),(\id\tensor \eps)(\eta)}_{A} 
= \ip{\xi,\eta}_{A}.
\]
Hence this map induces an isometry 
$V_E \in \cl{L}_A^{\G}((\ind_{\H}^{\G}E)_0,E)$. To see the surjectivity of 
$V_E$, we use $P_E$ as in the previous lemma. For any $\xi \in E$ and 
$x \in \cl{O}(\G)$, we have
\[
(\id\tensor \eps)\circ P_E(\xi\tensor x) 
= (\id\tensor h_{\H})((\id\tensor S)(\alpha_E(\xi))(1\tensor q(x))).
\]
Since 
$\aspan \alpha_E(E)(\C\tensor \cl{O}(\H)) = E\tensor \cl{O}(\H)$,
the above equality implies the surjectivity of $V_E$.

Next take $\tilde{E} \in \G\text{-}\mod_{\tilde{A}}$ arbitrarily. 
Then $\tilde{\alpha}_{\tilde{E}}$ defines a linear map from $\tilde{\cl{E}}$
to $\tilde{E}_0\tensor \cl{O}(\G)$, for which we have
\[
 \ip{\tilde{\alpha}_{\tilde{E}}(\xi),\tilde{\alpha}_{\tilde{E}}(\eta)}_{\tilde{A}} = \ip{\xi,\eta}_{\tilde{A}}.
\]
Moreover we also have
\begin{align*}
 \ip{\tilde{\alpha}_{\tilde{E}}(\xi),\tilde{\alpha}_{\tilde{E}}(\eta x)}_{\tilde{A}}
&= (\id\tensor \eps \tensor \id)(\tilde{\alpha}(\ip{\xi,\eta}_{\tilde{A}}x)) \\
&= (\id\tensor \eps \tensor \id)(\tilde{\alpha}(\ip{\xi,\eta}_{\tilde{A}}))x
= \ip{\tilde{\alpha}_{\tilde{E}}(\xi),\tilde{\alpha}_{\tilde{E}}(\eta)x}_{\tilde{A}}.
\end{align*}
Hence $\tilde{\alpha}_{\tilde{E}}$ induces an isometry 
$U_{\tilde{E}} \in \cl{L}_{\tilde{A}}^{\G}(\tilde{E},\tilde{E}_0\tensor C(\G))$. It can be easily seen that the image of $U_{\tilde{E}}$ is contained
in $\ind_{\H}^{\G}\tilde{E}_0$. To see that 
$U_{\tilde{E}}(\tilde{E}) = \ind_{\H}^{\G} \tilde{E}_0$, we can
use Lemma \ref{lemm:ind ev} by considering the case of 
$B = A$ and $M = A$.
\end{proof}
\begin{rema}
 The functor $\map{\mathrm{Ev}_{\G}^{\H}}{\G\text{-}\mod_{\tilde{A}}}{\H\text{-}\mod_A}$ also has a structure of $\rep^{\fin}\G$-module functor, namely
the collection of unitaries 
$v\tensor \xi \in (\cl{H}_{\pi}\tensor \tilde{E})_0
\longmapsto v\tensor \xi \in \cl{H}_{\pi|_{\H}}\tensor \tilde{E}_0$.
\end{rema}

Now let $A,B$ be $\H$-\cstar-algebras and 
$\tilde{A},\tilde{B}$ be their induced $\G$-\cstar-algebra.
Then we also have an induction functor 
$\map{\ind_{\H}^{\G}}{\H\text{-}\corr_{(A,B)}}{\G\text{-}\corr_{\tilde{A},\tilde{B}}}$.

\begin{lemm} \label{lemm:ind ev}
 Let $\tilde{E} \in \G\text{-}\mod_{\tilde{A}}$ and 
$M \in \H\text{-}\corr_{A,B}$.
Then we have the following unitary equivalence
\[
 \tilde{E}\tensor_{\tilde{A}}\ind_{\H}^{\G}M
\cong\ind_{\H}^{\G}(\mathrm{Ev}_{\G}^{\H}\tilde{E}\tensor_A M),
\] 
induced by $\xi\tensor \eta\tensor x \longmapsto 
\tilde{\alpha}_{\tilde{E},13}(\xi)(1\tensor \eta\tensor x)$.
\end{lemm} 
\begin{proof}
The only non-trivial part is the surjectivity of the map in the statement.
Let $\tilde{E},\cl{M}$ and $(\ev_{\G}^{\H}\tilde{E}\tensor_A M)_{\mathrm{alg}}$ be the algebraic cores of $\tilde{E},M$ and $\ev_{\G}^{\H}\tilde{E}\tensor_A M$ respectively.
Let us consider the following diagram:
\[
 \xymatrix
{
\tilde{\cl{E}}\tensor \cl{M}\tensor \cl{O}(\G) \ar[r] \ar[dd]_-{\id\tensor P_M} & \tilde{\cl{E}}\tensor \cl{M} \tensor \cl{O}(\G) \ar[d] \\
 & (\ev_{\G}^{\H}\tilde{E}\tensor_A M)_{\mathrm{alg}} \tensor \cl{O}(\G) \ar[d]^-{P_{\ev_{\G}^{\H}\tilde{E}\tensor_A M}} \\
\tilde{E}\tensor_{\tilde{A}}\ind_{\H}^{\G}M \ar[r] & \ind_{\H}^{\G}(\ev_{\G}^{\H}\tilde{E}\tensor_A M).
}
\]
Here the morphism at the top is given by
$\xi\tensor\eta\tensor x \longmapsto \tilde{\alpha}_{\tilde{E},13}(\xi)(1\tensor \eta\tensor x)$, and the morphism at the bottom is as in the statement.
Then we can see that this diagram is commutative and that
the morphisms at the top and at the right have dense ranges.
By using these observation, we can see
the surjectivity of the morphism at the bottom. 
\end{proof}
\begin{rema}
One should note that the collection of the unitary equivalences over all
$\tilde{E}$ gives rise to a unitary equivalence of $\rep^{\fin}\G$-module
functor from $\tend\tensor_{\tilde{A}}\ind_{\H}^{\G}M$ to 
$\ind_{\H}^{\G}(\ev_{\G}^{\H}(\tend)\tensor_A M)$
\end{rema}

Let $C$ be another $\H$-\cstar-algebra. For
$M \in \H\text{-}\corr_{A,B}$ and $N \in \H\text{-}\corr_{B,C}$,
we have a unitary equivalence 
$\ind_{\H}^{\G}(M\tensor_B N)\cong (\ind_{\H}^{\G}M)\tensor_{\tilde{B}} (\ind_{\H}^{\G}N)$, 
namely the restriction of
$(\xi\tensor \eta)\tensor x \longmapsto (\xi\tensor 1)\tensor (\eta\tensor x)$. 
Moreover the collection of these unitaries makes 
$\map{\ind_{\H}^{\G}}{\H\text{-}\corr_A}{\G\text{-}\corr_{\tilde{A}}}$ 
into a \cstar-tensor functor. 

\begin{thrm} \label{thrm:comparison}
Let $A,B$ be quantum homogeneous spaces of $\H$ and $\tilde{A},\tilde{B}$
be their induced quantum homogeneous spaces of $\G$. Then the following
diagram of \cstar-functors commutes up to a canonical unitary natural
transformation:
\[
 \xymatrix
{
\H\text{-}\corr^{\rfin}_{A,B} \ar[dd]_-{\ind_{\H}^{\G}} \ar[r]^-{\cong} & [\H\text{-}\mod^{\fin}_A,\H\text{-}\mod^{\fin}_B]^{\rep^{\fin}\H} \ar[d] \\
& [\H\text{-}\mod^{\fin}_A,\H\text{-}\mod^{\fin}_B]^{\rep^{\fin}\G} \ar[d]^-{\ind_{\H}^{\G}\circ\tend\circ \ev_{\G}^{\H}}_-{\cong} \\
\G\text{-}\corr^{\rfin}_{\tilde{A},\tilde{B}} \ar[r]^-{\cong} & [\G\text{-}\mod^{\fin}_{\tilde{A}}, \G\text{-}\mod^{\fin}_{\tilde{B}}]^{\rep^{\fin}\G}.
}
\]
Moreover, the canonical unitary natural transformation
is monoidal if $A = B$.
\end{thrm}
\begin{proof}
The canonical unitary natural transformation is given by the collection
of unitaries as in Lemma \ref{lemm:ind ev}. Then the statement can
be checked by direct calculations.
\end{proof}

\subsection{Tracial module category and its Plancherel weight}

Theorem \ref{thrm:comparison} enables us to show an imprimitivity-type 
result by a comparison of $\rep^{\fin}\G$-module 
functors and $\rep^{\fin}\H$-module functors. The purpose of 
this subsection is to build a general theory of such a comparison
for general module categories with traces.

At first we introduce the notion of a trace on a module category,
which already appears in 
{\cite[Definition 3.7]{Sc13}} for purely algebraic cases.
 
Let $\cl{C}$ be a rigid strict \cstar-tensor category and 
$\cl{M}$ be a $\cl{C}$-module category. The endomorphism algebra of 
$X \in \cl{M}$ is denoted by $\cl{M}(X)$. For $U \in \cl{C}$
and $X \in \cl{M}$, we can define the partial trace 
$\map{\tr_U\tensor \id}{\cl{M}(U\tensor X)}{\cl{M}(X)}$ as follows:
\[
 (\tr_U\tensor\id)(T) 
= (R_U^*\tensor\id_X)(\id_{\bar{U}}\tensor X)(R_U\tensor\id_X)
\]
Here $(U,\bar{U},R_U,\bar{R_U})$ is a standard solution in $\cl{C}$.
\begin{defn} \label{defn:module trace}
Let $\cl{M}$ be a $\cl{C}$-module category. A 
$\cl{C}$-module trace on $\cl{M}$ is a family 
$\{\map{\tr_X}{\cl{M}(X)}{\C}\}_{X \in \cl{M}}$ of positive 
linear maps satisfying the following conditions:
\begin{enumerate}
 \item For any $f,\,g \in \cl{M}(X,Y)$, $\tr_X(g^*f) = \tr_Y(fg^*)$.
 \item $\tr_{U\otimes X} = \tr_X\circ(\tr_U\otimes\id)$.
\end{enumerate}
We call a pair of $\cl{C}$-module category and $\cl{C}$-module trace on it 
as tracial $\cl{C}$-module category. In this case we define the 
dimension of $X \in \cl{M}$ as $d(X) = \tr_X(\id_X)$. 
\end{defn}
By (ii),
we have $d(U\tensor X) = d(U)d(X)$ for $U \in \cl{C}$ and $X \in \cl{M}$.
\begin{rema}
If $\cl{M}$ is indecomposable, there exists at most one
 $\cl{C}$-module trace on $\cl{M}$ up to scalar multiplication.
This can be seen by using the compatibility with the left $\cl{C}$-action. 
\end{rema}

\begin{exam}
If we consider $\cl{C}$ as a $\cl{C}$-module category in usual way,
then the family of categorical traces defines a $\cl{C}$-module trace.
Moreover, if another rigid strict \cstar-tensor category $\cl{D}$ and 
dimension-preserving \cstar-tensor functor $\map{F}{\cl{C}}{\cl{D}}$
are given, $\cl{D}$ can be considered as a tracial 
$\cl{C}$-module category by $U\tensor X = F(U)\tensor X$ and 
the family of categorical traces.  
\end{exam}

\begin{exam}
Let $A$ be a Q-system in $\cl{C}$. Then the category $\mod\text{-}A$ of 
right $A$-modules in $\cl{C}$ has a $\cl{C}$-module trace.
\end{exam}

\begin{exam}
If $\cl{C}$ is a fusion category, every $\cl{C}$-module category has 
a $\cl{C}$-module trace. This follows from the fact 
that every indecomposable $\cl{C}$-module category arises 
as a corner of a rigid \cstar-$2$-category with $\cl{C}$ in its diagonal.
\end{exam}

Let $\cl{M}$ be a tracial $\cl{C}$-module category. Fix a set of 
representatives of all equivalence classes of irreducible objects of
$\cl{M}$. It is denoted by $\irr\cl{M}$.

For $F \in [\cl{M},\cl{M}]_{\mathrm{b}}$, the algebra of bounded natural 
transformations from $F$ to $F$
is denoted by $\End_{\mathrm{b}}(F)$. Then 
$\End_{\mathrm{b}}(\id_{\cl{M}})$ is isomorphic 
to $\ell^{\infty}(\irr\cl{M})$ via a map 
$a \longmapsto (a_i)_{i \in \irr{\cl{M}}}$.
By using the dimension of objects in $\cl{M}$,
we can define a weight $\omega_{\cl{M}}$ on 
$\End_{\mathrm{b}}(\id_{\cl{M}})$ as follows:
\[
 \omega_{\cl{M}}(a) = \sum_{i \in \irr\cl{M}} d(i)^2 a_i.
\] 
We call $\omega_{\cl{M}}$ the \emph{Plancherel weight} of $\cl{M}$.
This defines a linear functional if and only if $\irr\cl{M}$ is 
a finite set. 

Recall we have a \cstar-tensor functor 
$\map{\Phi}{\cl{C}}{[\cl{M},\cl{M}]_{\mathrm{b}}}$
induced by the left action of $\cl{C}$ on $\cl{M}$.
Namely $\Phi(U)$ is a \cstar-fuctor from $\cl{M}$ to $\cl{M}$
which sends $X \in \cl{M}$ to $U\tensor X \in \cl{M}$.
We use the symbol $\phi_{U,V}$ to denote the unitary 
from $\Phi(U)\tensor \Phi(V)$ to $\Phi(U\tensor V)$.
We can show the following nice property of $\omega_{\cl{M}}$.

\begin{prop} \label{prop:trace}
Let $\cl{M}$ be a tracial $\cl{C}$-module category and $\omega_{\cl{M}}$ 
be the Plancherel weight. Then we have the following equalities for 
any $U \in \cl{C}$ and a positive natural transformation 
$\map{\eta}{\Phi(U)}{\Phi(U)}$:
\[
 \omega_{\cl{M}}(\Phi(R_U)^*\phi_{\bar{U},U}(\id\otimes\eta)
\phi_{\bar{U},U}^*\Phi(R_U)) 
= \omega_{\cl{M}}(\Phi(\bar{R_U})^*\phi_{U,\bar{U}}(\eta\otimes\id)
\phi_{U,\bar{U}}^*\Phi(\bar{R_U})).
\]
\end{prop}

To prove this proposition, it is convenient to replace 
$[\cl{M},\cl{M}]_{\mathrm{b}}$ by the category of column-finite
bi-graded Hilbert spaces introduced in \cite[Notation A.3.1]{DY13}.

Set $I = \irr\cl{M}$. Then the category $\hilb_{I\times I}$
of $I\times I$-graded Hilbert spaces has a canonical structure
of a \cstar-multitensor category. Namely, for $\cl{H} = (\cl{H}_{ij})_{ij}$
and $\cl{K} = (\cl{K}_{ij})_{ij}$, their tensor product 
$\cl{H}\tensor \cl{K}$ is a $I\times I$-graded Hilbert space whose
$(i,j)$-component is given by
\[
 \bigoplus_{k \in I} \cl{H}_{ik}\tensor\cl{K}_{kj}.
\]
A $I\times I$-graded Hilbert space $\cl{H}$ is said to be  
\emph{column-finite} if, for each $j \in I$, there are finitely many
$i \in I$ such that $\cl{H}_{ij}\neq 0$. We also say that $\cl{H}$ is 
\emph{uniformly finite} if it satisfies the following:
\begin{align*}
 \sup_{i \in I} \sum_{j \in I} \dim \cl{H}_{ij} < \infty,\quad
\sup_{j \in I} \sum_{i \in I} \dim \cl{H}_{ij} < \infty.
\end{align*}
The full subcategory consisting of column-finite (resp. uniformly finite)
$I\times I$-graded Hilbert spaces is denoted by 
$\hilb^{\mathrm{cf}}_{I\times I}$ (resp. $\hilb^{\fin}_{I\times I}$).
Then the rigid part of $\hilb_{I\times I}$ is precisely 
$\hilb^{\fin}_{I\times I}$ ({\cite[Lemma A.3.2]{DY13}}).

Now we compare $[\cl{M},\cl{M}]_{\mathrm{b}}$ and
$\hilb^{\mathrm{cf}}_{I\times I}$. For 
$F \in [\cl{M},\cl{M}]_{\mathrm{b}}$, we have a column-finite
$I\times I$-graded Hilbert space $\cl{H}_F$ whose $(i,j)$-component 
is given by $\cl{M}(i,F(j))$. Here $\cl{M}(i,F(j))$ is regarded 
as a Hilbert space by the inner product $\ip{S,T} = S^*T$. Then 
this construction gives an
equivalence of \cstar-multitensor categories from 
$[\cl{M},\cl{M}]_{\mathrm{b}}$ to $\hilb^{\mathrm{cf}}_{I\times I}$
({\cite[Theorem A.2.1, Theorem A.2.2]{DY13}}). Moreover we also have
an equivalence 
$[\cl{M},\cl{M}]^{\fin}_{\mathrm{b}} \cong \hilb^{\fin}_{I\times I}$ 
({\cite[Proposition A.3.3]{DY13}}).
From now on, we identify these categories by this equivalence.

Let $(U,\bar{U},R,\bar{R})$ be a standard solution in $\cl{C}$.
Then $\phi_{\bar{U},U}^*\Phi(R)$ gives a morphism from $\mbf{1}$ to 
$\cl{H}_{\Phi(\bar{U})}\tensor \cl{H}_{\Phi(U)}$. By the definition
of a tensor product of $I\times I$-graded Hilbert spaces, this morphism 
can be presented as a direct sum of
\[
\map{R_{ji}}{\C}{\cl{M}(i,\bar{U}\tensor j)\tensor \cl{M}(j,U\tensor i)},
\]
where $i,\,j \in I$.
Similarly $\phi_{U,\bar{U}}^*\Phi(\bar{R})$ is presented as a direct sum of 
\[
 \map{\bar{R}_{ji}}{\C}{\cl{M}(j,U\tensor i)\tensor\cl{M}(i,\bar{U}\tensor j)}.
\]

In the following lemma, $U\tensor i$ is denoted by $Ui$.

\begin{lemm}
For any $i,j \in I$, 
$(\cl{M}(j,Ui),\cl{M}(i,\bar{U}j), R_{ji},\bar{R}_{ji})$ 
is a solution of a conjugate equation in $\hilb^{\fin}$.
Moreover, if we replace $R_{ji}$ and $\bar{R}_{ji}$ with 
$d(i)^{1/2}d(j)^{-1/2}R_{ji}, d(i)^{-1/2}d(j)^{1/2}\bar{R}_{ji}$
respectively, it gives a standard solution.
\end{lemm}
\begin{proof}
 Since 
$(\Phi(U),\Phi(\bar{U}),\phi_{\bar{U},U}^*\Phi(R),\phi_{U,\bar{U}}\Phi(\bar{R}))$ is a solution of conjugate equation in $[\cl{M},\cl{M}]_{\mathrm{b}}$,
the following composition of maps is the identity:
\begin{align*}
\bigoplus_{j,i \in I} \cl{M}(j,Ui) &\xrightarrow{\bigoplus_{i,k} \bigoplus_{j=l} \bar{R}_{jk}\tensor\id} \bigoplus_{j,i \in I}\bigoplus_{k,l \in I} \cl{M}(j,Uk)\tensor \cl{M}(k,\bar{U}l)\tensor \cl{M}(l,Ui) \\
&\xrightarrow{\bigoplus_{j,l}\bigoplus_{i=k} \id\tensor R_{li}^*} \bigoplus_{j,i \in I} \cl{M}(j,Ui).
\end{align*}
Hence we have $(\id\tensor R_{ji}^*)(\bar{R}_{ji}\tensor \id) = \id$
for any $i,j \in I$. The same argument also shows the other
equality $(\id\tensor \bar{R}_{ji^*})(R_{ji}\tensor\id) = \id$.
Now we complete the proof of the former half of the statement.

To prove the latter half, we calculate $\nor{R_{ji}}$ and 
$\nor{\bar{R}_{ji}}$. At first one should note that 
$\cl{M}(i,\bar{U}j)\tensor\cl{M}(j,Ui)$ can be embedded in
$\cl{M}(i, \bar{U}Ui)$ by $T\tensor S \longmapsto (\id_j\tensor S)T$.
Moreove the projection $p_j$ onto its image is given by the following:
\[
 p_j(X) = \sum_{V \in (j,Ui)} (\id\tensor VV^*)X.
\]
Then $R_{ji}(1)$ corresponds to the image by $p_j$ of the component of 
$\phi_{\bar{U},U}^*\Phi(R)$ at $i \in I$.
Hence we have
\begin{align*}
 \nor{R_{ji}}^2 
&= \sum_{V \in (j,Vi)} (R\tensor\id_i)^*(\id\tensor VV^*)(R\tensor\id_i)\\
&= \frac{1}{d(i)}\sum_{V \in (j,Vi)} \tr_{Ui}(VV^*) \\
&= \frac{1}{d(i)}\sum_{V \in (j,Vi)} \tr_j(V^*V) \\
&= \frac{d(j)}{d(i)}\dim \cl{M}(j,Ui).
\end{align*}
Similarly we also have
\[
 \nor{\bar{R}_{ji}}^2 
= \frac{d(i)}{d(j)}\dim \cl{M}(i,\bar{U}j)
= \frac{d(i)}{d(j)}\dim \cl{M}(j,Ui).
\]
Hence we have
$\nor{R_{ji}}\nor{\bar{R}_{ji}} = (\dim\cl{M}(j,Ui))^2$.
By using the characterization of standard solutions, we get the conclusion.
\end{proof}

\begin{proof}[Proof of Proposition \ref{prop:trace}]
By the previous lemma, we have
\[
 R_{ji}(\id\tensor \theta)R_{ji}^* = \frac{d(j)}{d(i)}\tr (\theta)
\]
for any linear map $\map{\theta}{\cl{M}(j,Ui)}{\cl{M}(j,Ui)}$.

Let $\map{\eta}{\Phi(U)}{\Phi(U)}$ be a positive natural transformation.
Then it induces a positive operator
$\map{\eta_{ji}}{\cl{M}(j,Ui)}{\cl{M}(j,Ui)}$ for any $i,j \in I$.
Then we have 
\begin{align*}
 \omega_{\cl{M}}(\Phi(R_U)^*\phi_{\bar{U},U}(\id\otimes\eta)
\phi_{\bar{U},U}^*\Phi(R_U)) 
&= \sum_{i\in I}d(i)^2 \sum_{j \in I} R_{ji}^*(\id\tensor \eta_{ji})R_{ji} \\
&= \sum_{i,j \in I} d(i)d(j)\tr(\eta_{ji}).
\end{align*}
For the right hand side of the equality in the statement, we also can
show that
\[
 \omega_{\cl{M}}(\Phi(\bar{R}_U)^*\phi_{U,\bar{U}}(\eta\tensor\id)
\phi_{U,\bar{U}}^*\Phi(\bar{R}_U)) 
= \sum_{i,j \in I} d(i)d(j)\tr(\eta_{ji}). 
\]
Hence the statement holds.
\end{proof}

Based on Proposition \ref{prop:trace}, we introduce a notion of 
a standard solution in $[\cl{M},\cl{M}]_{\mathrm{b}}$.

\begin{defn}
Let $\map{F}{\cl{M}}{\cl{M}}$ be a \cstar-functor. We say that a solution
$(F,\bar{F},R,\bar{R})$ is \emph{standard} if it satisfies the following
identity for any positive natural transformation $\map{\eta}{F}{F}$:
\[
 \omega_{\cl{M}}(R^*(\id\tensor \eta)R) 
= \omega_{\cl{M}}(\bar{R}^*(\eta\tensor\id)\bar{R}).
\]
In this case, we define a weight $\tr_F$ on $\mathrm{End}_{\mathrm{b}}(F)$
by $\eta \longmapsto \omega_{\cl{M}}(R^*(\id\tensor \eta)R)$.
\end{defn}

Proposition \ref{prop:trace} states that 
$(\Phi(U),\Phi(\bar{U}),\phi_{\bar{U},U}^*\Phi(R_U),\phi_{U,\bar{U}}^*\Phi(\bar{R}_U))$ 
is standard when $(U,\bar{U},R,\bar{R})$ is standard.

\begin{lemm} \label{lemm:standard}
Let $F \in [\cl{M},\cl{M}]_{\mathrm{b}}$ be a rigid object.
\begin{enumerate}
 \item A standard solution for $F$ exists and is unique up 
to unitary equivalence.
 \item $\tr_F$ is tracial and independent of a choice of a standard solution.
 \item For a standard solution $(F,\bar{F},R,\bar{R})$, the following map
is an anti-\star-isomorphism from $\End_{\mathrm{b}}(F)$ to $\End_{\mathrm{b}}(\bar{F})$:
\[
 \eta \longmapsto (\id\tensor \bar{R}^*)(\id\tensor\eta\tensor\id)(R\tensor\id).
\]
\end{enumerate}
\end{lemm}
\begin{proof}
If we regard $F$ as an object $\cl{H}_F \in \hilb^{\fin}_{I\times I}$,
a solution $(F,\bar{F},R,\bar{R})$ of conjugate equation is decomposed into
a family $\{(\cl{M}(j,F(i)),\cl{M}(i,\bar{F}(j)),R_{ji},\bar{R}_{ji})\}$
of solutions in $\hilb^{\fin}$. Then we have the following for any
positive natural transformation $\map{\eta}{F}{F}$: 
\begin{align*}
 \omega_{\cl{M}}(R^*(\id\tensor\eta)R) 
&= \sum_{i,j \in I} d(i)^2R_{ji}^*(\id\tensor \eta_{ji})R_{ji},\\
 \omega_{\cl{M}}(\bar{R}^*(\eta\tensor\id)\bar{R})
&= \sum_{i,j \in I} d(j)^2\bar{R}^*_{ji}(\eta_{ji}\tensor\id)\bar{R}_{ji},
\end{align*}
where $\eta_{ji}$ is a corresponding positive operator on $\cl{M}(j,F(i))$.

Hence $(F,\bar{F},R,\bar{R})$ is standard if and only if
the following quadruple is a standard solution for each $i,j \in I$:
\[
 \left(\cl{M}(j,F(i)),\cl{M}(i,\bar{F}(j)),d(j)^{1/2}d(i)^{-1/2}R_{ji},d(j)^{-1/2}d(i)^{1/2}\bar{R}_{ji}\right).
\]
Now all of the statements follow from this characterization.
\end{proof}

We would like to come back to our main interest: a comparison 
of module functors. 
Let $\cl{D}$ be another rigid strict \cstar-tensor category and 
$\map{(\Theta,\theta)}{\cl{D}}{\cl{C}}$ be a \cstar-tensor functor.

\begin{defn}
The modular natural transformation $a_{\Theta}$ of $(\Theta,\theta)$ is 
a natural transformation 
from $\Theta$ to $\Theta$ given by the following:
\[
 a_{\Theta,U} 
= (\id\tensor\tr_{\Theta(\bar{U})})(\theta_{U,\bar{U}}^*\Theta(\bar{r}_U\bar{r}_U^*)\theta_{U,\bar{U}}).
\]
Here $(U,\bar{U},r_U,\bar{r}_U)$ is a stardard solution in $\cl{D}$.
\end{defn}

Since all standard solutions are mutually unitary equivalent, 
$a_{\Theta,U}$ does not depend on the choice of $(U,\bar{U},r_U,\bar{r}_U)$.

We collect some fundamental properties of $a_{\Theta}$ 
in the following lemma.

\begin{lemm} \label{lemm:modular auto}
The modular natural transformation $a_{\Theta}$ is monoidal and invertible.
Its inverse is as follows:
\[
 a_{\Theta,U}^{-1} = (\tr_{\Theta(\bar{U})}\tensor\id)(\theta_{\bar{U},U}^*\Theta(r_{U}r_{U}^*)\theta_{\bar{U},U}).
\]
Moreover, the following quadruple is a standard
solution in $\cl{C}$ for every $U \in \cl{D}$:
\[
 (\Theta(U),\Theta(\bar{U}),(1\tensor a_{\Theta,U}^{1/2})\theta_{\bar{U},U}^*\Theta(r_U),(a_{\Theta,U}^{-1/2}\tensor\id)\theta_{U,\bar{U}}^*\Theta(\bar{r}_U)).
\]
\end{lemm}
\begin{proof}
One can check the statement by direct calculations.
\end{proof}
\begin{exam}
If $(\Theta,\theta)$ is the fiber functor from 
$\rep^{\fin}\G$ to $\hilb^{\fin}$,
each component $a_{\Theta,\pi}$ is precisely $\rho_{\pi}$ in 
{\cite[Proposition 1.4.4]{NT13}}. See also {\cite[Example 2.2.13]{NT13}}. 
\end{exam}

We regard a $\cl{D}$-module functor $(F,f)$ as a pair of a \cstar-functor 
$F$ and a collection of unitary morphisms 
$\map{f_U}{F\tensor \Phi(\Theta(U))}{\Phi(\Theta(U))\tensor F}$.
 
\begin{prop} \label{prop:abstract descent}
Let $\map{(F,f)}{\cl{M}}{\cl{M}}$ be a $\cl{D}$-module functor. If 
$F$ is rigid in $[\cl{M},\cl{M}]_{\mathrm{b}}$, 
the following are equivalent for each $U \in \cl{D}$.
\begin{enumerate}
 \item We have $f_U(\id_F\tensor \Phi(a_{\Theta,U})) = (\Phi(a_{\Theta,U})\tensor \id_F)f_U$.
 \item For a standard solution 
$(\Phi(\Theta(U)),\bar{\Phi(\Theta(U))},R_{\Phi(\Theta(U))},\bar{R}_{\Phi(\Theta(U))})$ in $[\cl{M},\cl{M}]_{\mathrm{b}}$,
the following natural transformation is unitary:
\[
 (\id\tensor\id\tensor \bar{R}_{\Phi(\Theta(U))}^*)(\id\tensor f_{U}^*\tensor\id)(R_{\Phi(\Theta(U))}\tensor\id\tensor\id).
\]
 \item For a standard solution $(F,\bar{F},R,\bar{R})$, the following natural transformation $\map{\bar{f}_U}{\bar{F}\tensor \Phi(\Theta(U))}{\Phi(\Theta(U))\tensor \bar{F}}$ is unitary:
\[
 \bar{f}_U = (R^*\tensor \id\tensor\id)(\id\tensor f_U^*\tensor \id)(\id\tensor\id\tensor \bar{R}).
\]
\end{enumerate}
\end{prop}

\begin{proof}
Let $(U,\bar{U},r_U,\bar{r}_U)$ be a standard solution in $\cl{D}$
and define $r_{\Phi(\Theta(U))}$ and $\bar{r}_{\Phi(\Theta(U))}$
as follows:
\begin{align*}
r_{\Phi(\Theta(U))} &= \phi^*_{\bar{\Theta(U)},\Theta(U)}\Phi(\theta_{\bar{U},U}^*\Theta(r_U)),\\
\bar{r}_{\Phi(\Theta(U))} &= \phi^*_{\Theta(U),\bar{\Theta(U)}}\Phi(\theta_{U,\bar{U}}^*\Theta(\bar{r}_{U})).
\end{align*}
By using the braiding equation, we have 
\begin{align*}
& (\id\tensor\id\tensor \bar{r}_{\Phi(\Theta(U))}^*)(\id\tensor f_U^*\tensor \id)(r_{\Phi(\Theta(U))}\tensor\id\tensor\id)f_{\bar{U}}^* \\
&= (\id\tensor\id\tensor \bar{r}_{\Phi(\Theta(U))}^*)(\id\tensor f_{U\tensor \bar{U}}^*)(r_{\Phi(\Theta(U))}\tensor\id\tensor\id) \\
&= (\id\tensor\bar{r}_{\Phi(\Theta(U))}^*\tensor\id)(r_{\Phi(\Theta(U))}\tensor\id\tensor\id) \\
&= \id\tensor\id.
\end{align*}
Since $f_{\bar{U}}$ is unitary, 
this implies that we have
\begin{align*}
f_{\bar{U}} 
&= (\id\tensor\id\tensor \bar{r}_{\Phi(\Theta(U))}^*)(\id\tensor f_U^*\tensor \id)(r_{\Phi(\Theta(U))}\tensor\id\tensor\id),\\
f_U 
&= (\id\tensor\id\tensor r_{\Phi(\Theta(U))}^*)(\id\tensor f_{\bar{U}}^*\tensor \id)(\bar{r}_{\Phi(\Theta(U))}\tensor\id\tensor\id).
\end{align*}
Now we give a proof of (i) $\iff$ (ii). 
Let $(\Theta(U),\Theta(\bar{U}),R_{\Theta(U)},\bar{R}_{\Theta(U)})$
be the standard solution as in Lemma \ref{lemm:modular auto}. Set
\begin{align*}
 R_{\Phi(\Theta(U))}
&= \phi_{\Theta(\bar{U}),\Theta(U)}^*\Phi(R_{\Theta(U)}),\\
 \bar{R}_{\Phi(\Theta(U))}
&= \phi_{\Theta(U),\Theta(\bar{U})}^*\Phi(\bar{R}_{\Theta(U)}).
\end{align*}
Then Proposition \ref{prop:trace} asserts that $(\Phi(\Theta(U)),\Phi(\Theta(\bar{U})),R_{\Phi(\Theta(U)))},\bar{R}_{\Phi(\Theta(U))})$ is a
standard solution. Under the condition (i), we have
\begin{align*}
& (\id\tensor\id\tensor \bar{R}_{\Phi(\Theta(U))}^*)(\id\tensor f_{U}^*\tensor\id)(R_{\Phi(\Theta(U))}\tensor\id\tensor\id) \\
&= (\id\tensor\id\tensor \bar{R}_{\Phi(\Theta(U))}^*)(\id\tensor (\id\tensor \Phi(a_{\Theta,U}^{1/2}))f_{U}^*(\Phi(a_{\Theta,U}^{-1/2})\tensor\id)\tensor\id)(R_{\Phi(\Theta(U))}\tensor\id\tensor\id) \\
&=(\id\tensor\id\tensor \bar{r}_{\Phi(\Theta(U))}^*)(\id\tensor f_U^*\tensor \id)(r_{\Phi(\Theta(U))}\tensor\id\tensor\id) \\
&= f_{\bar{U}}.
\end{align*}
Hence (ii) holds for the standard solution. Since all standard solutions
are mutually unitary equivalent, we also have (ii) for a general standard
solution. 
To show the converse direction, one should note that we have
\begin{align*}
 (\id\tensor a_{\Theta,U}^{1/2})\theta_{\bar{U},U}^*\Theta(r_U)
&=(a_{\Theta,\bar{U}}^{-1/2}\tensor\id)\theta_{\bar{U},U}^*a_{\Theta,\bar{U}\tensor U}^{1/2}\Theta(r_U) \\
&= (a_{\Theta,\bar{U}}^{-1/2}\tensor\id)\theta_{\bar{U},U}^*\Theta(r_U)a_{\mbf{1}}^{1/2}
=(a_{\Theta,\bar{U}}^{-1/2}\tensor\id)\theta_{\bar{U},U}^*\Theta(r_U).
\end{align*}
Similary we have
\begin{align*}
 (\id\tensor a_{\Theta,U}^{1/2})\theta_{U,\bar{U}}^*\Theta(\bar{r}_U)
=(a_{\Theta,\bar{U}}^{-1/2}\tensor\id)\theta_{U,\bar{U}}^*\Theta(\bar{r}_U).
\end{align*}
By using these formulae, we can see that
\begin{align*}
&(\Phi(a_{\Theta,\bar{U}}^{-1/2})\tensor\id)f_{\bar{U}}(\id\tensor \Phi(a_{\Theta,\bar{U}}^{1/2}))\\
&=(\id\tensor\id\tensor \bar{r}_{\Phi(\Theta(U))}^*)(\Phi(a_{\Theta,\bar{U}}^{-1/2})\tensor f_{U}^*\tensor \Phi(a_{\Theta,\bar{U}}^{1/2}))(r_{\Phi(\Theta(U))}\tensor\id\tensor\id) \\
&=(\id\tensor\id\tensor \bar{R}_{\Phi(\Theta(U))}^*)(\id\tensor f_{U}^*\tensor\id)(R_{\Phi(\Theta(U))}\tensor\id\tensor\id).
\end{align*}
Hence $(\Phi(a_{\Theta,\bar{U}}^{-1/2})\tensor\id)f_{\bar{U}}(\id\tensor \Phi(a_{\Theta,\bar{U}}^{1/2}))$ is unitary. Now we have
\[
 (\Phi(a_{\Theta,\bar{U}}^{-1/2})\tensor\id)f_{\bar{U}}(\id\tensor \Phi(a_{\Theta,\bar{U}}))f_{\bar{U}}^*(\Phi(a_{\Theta,\bar{U}}^{-1/2})\tensor\id) = 1.
\]
Then we can see that (i) holds for $\bar{U}$ since $f_{\bar{U}}$ is unitary.
We already have shown (i) $\Longrightarrow$ (ii) and (ii) for $U \Longrightarrow$ (i) for $\bar{U}$, hence (ii) $\Longrightarrow$ (i) also has been shown.

The equivalence of (ii) and (iii) can be seen from Lemma \ref{lemm:standard}
 (iii) and the standardness of the following quadruples:
\begin{align*}
 (F\tensor \Phi(\Theta(U)),\Phi(\bar{\Theta(U)})\tensor\bar{F},(\id\tensor R\tensor\id)R_{\Phi(\Theta(U))},(\id\tensor \bar{R}\tensor\id)\bar{R}_{\Phi(\Theta(U))}),\\
 (\Phi(\Theta(U))\tensor F, \bar{F}\tensor \Phi(\bar{\Theta(U)}),(\id\tensor R_{\Phi(\Theta(U))}\tensor\id)R,(\id\tensor \bar{R}_{\Phi(\Theta(U))}\tensor\id)\bar{R}).
\end{align*}
\end{proof}

By considering the case of $\map{\id_{\cl{C}}}{\cl{C}}{\cl{C}}$,
we obtain the following corollary.

\begin{coro}
For a $\cl{C}$-module functor $(F,f)$, it is rigid in 
$[\cl{M},\cl{M}]^{\cl{C}}_{\mathrm{b}}$ if and only if
$F$ is rigid in $[\cl{M},\cl{M}]_{\mathrm{b}}$. In this case, the conjugate
object $(\bar{F},\bar{f})$ of $(F,f)$ can be obtained from a standard 
solution $(F,\bar{F},R,\bar{R})$ as follows:
\begin{itemize}
 \item $\bar{F}$: a conjugate object of $F$ in 
$[\cl{M},\cl{M}]_{\mathrm{b}}$.
 \item $\bar{f}_U = (R^*\tensor\id\tensor\id)(\id\tensor f_U^*\tensor\id)(\id\tensor\id\tensor\bar{R})$. 
\end{itemize} 
\end{coro}
\begin{coro}
If $\cl{M}$ is connected as a tracial $\cl{C}$-module category,
a standard solution in $[\cl{M},\cl{M}]_{\mathrm{b}}^{\cl{C}}$
is also standard in $[\cl{M},\cl{M}]_{\mathrm{b}}$.
\end{coro}

For a $\cl{C}$-module category, we can consider an equivalence relation
$\sim_{\cl{C}}$ on $\irr\cl{M}$ as follows:
\[
 i \sim_{\cl{C}} j \overset{\mathrm{def}}{\iff} \text{there exists }U \in \cl{C}\text{ such that }\cl{M}(i,U\tensor j) \neq 0. 
\]

For a \cstar-functor $\map{F}{\cl{M}}{\cl{M}}$, we define a
function $\map{d_F}{\irr\cl{M}}{\R}$ as follows:
\[
 d_F(i) = \frac{d(F(i))}{d(i)}.
\]

\begin{prop} \label{prop:locally scaling}
If $(F,f) \in [\cl{M},\cl{M}]_{\mathrm{b}}^{\cl{D}}$ satisfies 
the conditions in Proposition \ref{prop:abstract descent} for all 
$U \in \cl{D}$, the function $d_F$ is constant on each equivalence class of 
$\sim_{\cl{D}}$.

If $F$ is an equivalence of categories, the converse direction also holds.
\end{prop}
\begin{proof}
 Let $(F,\bar{F},R,\bar{R})$ be a standard solution in 
$[\cl{M},\cl{M}]_{\mathrm{b}}$. Then the condition (iii) of Proposition
\ref{prop:abstract descent} implies that we have a $\cl{D}$-module functor
$\map{(\bar{F},\bar{f})}{\cl{M}}{\cl{M}}$ given by
\[
 \bar{f}_U = (R^*\tensor\id\tensor\id)(\id\tensor f_U^*\tensor\id)(\id\tensor\id\tensor \bar{R}).
\]
Then $R$ and $\bar{R}$ are morphisms in $[\cl{M},\cl{M}]_{\mathrm{b}}^{\cl{D}}$, hence $R^*R$ is also a morphism from $\id_{\cl{M}}$ to $\id_{\cl{M}}$ 
in $[\cl{M},\cl{M}]_{\mathrm{b}}^{\cl{D}}$. This implies that we have
\[
 (R^*R)_{\Theta(U)\tensor i} 
= \id_{\Theta(U)}\tensor (R^*R)_i
\]
for any $U \in \cl{D}$ and $i \in \irr\cl{M}$. In particular
$i \in \irr\cl{M} \longmapsto (R^*R)_i$ is constant on 
each equivalence class of $\sim_{\cl{D}}$. On the other hand, by the 
proof of Lemma \ref{lemm:standard}, we have
\[
 (R^*R)_i 
= \sum_{j \in \irr\cl{M}} \frac{d(j)\dim\cl{M}(j,F(i))}{d(i)}
= \frac{d(F(i))}{d(i)}.
\]
This completes a proof of the first statement.

To show the second statement, we will check the condition (iii) of 
Proposition \ref{prop:abstract descent}. Let $(F,\bar{F},R,\bar{R})$
be a standard solution. Our assumption implies that $F(i)$ is 
irreducible for any $i \in \irr\cl{M}$, hence 
$d(F(i))^{-1/2}d(i)^{1/2}R_i = d_F(i)^{-1/2}R_i$ is unitary.
Moreover $d_F(i)^{-1/2}R_{\Theta(U)\tensor i}$ is also unitary for 
any $U \in \cl{D}$ by the assumption on $d_F$. Since we have
\[
 d_F(\bar{F}(i)) = d(F(\bar{F}(i)))/d(\bar{F}(i)) = d_{\bar{F}}(i)^{-1},
\]
the similar argument shows that
$d_{\bar{F}}(i)^{-1/2}\bar{R}_{\Theta(U)\tensor i}$ is also unitary.

Then we have
\begin{align*}
 (R\tensor\id\tensor\id)&(\id\tensor f_U^*\tensor\id)(\id\tensor\id\tensor \bar{R})_i\\
&=d_F(\bar{F}(i))^{-1/2}R_{U\tensor \bar{F}(i)}\bar{F}(f_{U,\bar{F}(i)}^*)\bar{F}(\id_{\Theta(U)}\tensor d_{\bar{F}}(i)^{-1/2}\bar{R}_i).
\end{align*}
This implies that $(F,f)$ fulfills the condition (iii).
\end{proof}

We end this subsection by showing that
the condition (iii) is automatically satisfied under 
a finiteness condition.

\begin{thrm} \label{thrm:finite descent}
Let $\cl{M}$ be a tracial $\cl{C}$-module category with 
$\abs{\irr\cl{M}} < \infty$.
Then $[\cl{M},\cl{M}]^{\cl{D}}_{\mathrm{b}}$ is rigid.
Moreover a standard solution in $[\cl{M},\cl{M}]^{\cl{D}}_{\mathrm{b}}$
is standard again in $[\cl{M},\cl{M}]_{\mathrm{b}}$.
\end{thrm}

We need the following lemma to show the theorem.

\begin{lemm}
Let $(F,f)$ be a $\cl{D}$-module functor from $\cl{M}$ to $\cl{M}$ and
$\bar{F}$ be a conjugate of $F$ in $[\cl{M},\cl{M}]_{\mathrm{b}}$.
Then the function $d_F$ and $d_{\bar{F}}$ are constant on each 
equivalence class of $\sim_{\cl{D}}$.
\end{lemm}
\begin{proof}
Let $I$ be an equivalence class of $\sim_{\cl{D}}$. Since $\irr\cl{M}$
is finite, $I$ is also finite. Hence we can take an object $U \in \cl{D}$
such that $\cl{M}(i,\Theta(U)\tensor j) \neq 0$ for any $i,j \in I$.
This means the matrix $M = (\dim \cl{M}(j,\Theta(U)\tensor i))_{ij \in I}$
is an irreducible matrix with positive entries. Now $d = (d(i))_{i \in I}$
satisfies $Md = d(\Theta(U))d$ i.e. $d$ is a Perron-Frobenius eigenvector
of $M$. On the other hand, we have
\begin{align*}
d(\Theta(U))d(F(i))
&= d(\Theta(U)\tensor F(i))\\
&= d(F(\Theta(U)\tensor i))
= \sum_{j \in I} d(F(j))\dim \cl{M}(j,\Theta(U)\tensor i).
\end{align*}
Hence $(d(F(i)))_{i \in I}$ is also a Perron-Frobenius eigenvector of $M$,
which must be a scalar multiple of $d$. Now we get the statement for $d_F$.
For $\bar{F}$, one should note that $\bar{F}$ is an adjoint functor of $F$.
Hence we have
\begin{align*}
 \cl{M}(i,\bar{F}(\Theta(U)\tensor j)) 
&\cong \cl{M}(F(i),\Theta(U)\tensor j)\\
&\cong \cl{M}(\Theta(\bar{U})\tensor F(i), j)\\
&\cong \cl{M}(F(\Theta(\bar{U})\tensor i),j)\\
&\cong \cl{M}(\Theta(\bar{U})\tensor i,\bar{F}(j))
\cong \cl{M}(i,\Theta(U)\tensor \bar{F}(j)).
\end{align*}
for any $i,j \in \irr\cl{M}$. This implies 
$\bar{F}(\Theta(U)\tensor i)\cong \Theta(U)\tensor \bar{F}(i)$ and we 
can use the argument above to show the statement for $\bar{F}$.
\end{proof}

\begin{proof}[Proof of Theorem \ref{thrm:finite descent}]
Take $(F,f) \in [\cl{M},\cl{M}]^{\cl{D}}_{\mathrm{b}}$ and a standard
solution $(F,\bar{F},R,\bar{R})$ in $[\cl{M},\cl{M}]_{\mathrm{b}}$.
Since $(R^*R)_i = d_F(i)^{-1}$ and 
$(\bar{R}^*\bar{R})_i = d_{\bar{F}}(i)^{-1}$, the previous lemma
implies that
\[
R^*R\tensor \id_{\Phi(\Theta(U))} 
= \id_{\Phi(\Theta(U))}\tensor R^*R,\quad
\bar{R}^*\bar{R}\tensor\id_{\Phi(\Theta(U))} 
= \id_{\Phi(\Theta(U))}\tensor \bar{R}^*\bar{R}
\]
for any $U \in \cl{D}$.

Set $\bar{f}_U$ as in Proposition \ref{prop:abstract descent}:
\[
 \bar{f}_U 
= (R^*\tensor\id\tensor\id)(\id\tensor f_U^*\tensor\id)(\id\tensor\id\tensor\bar{R}).
\]
It suffices to show that $\bar{f}_U$ is unitary. Take
a standard solution $(U,\bar{U},r_U,\bar{r}_U)$ in $\cl{D}$
and set $r_{\Phi(\Theta(U))}$ and $\bar{r}_{\Phi(\Theta(U))}$
as in the proof of Proposition \ref{prop:abstract descent}.
Now consider 
$\omega_l \in \End_{\mathrm{b}}(\bar{F}\tensor \Phi(\Theta(U)))^*$
given by
\[
 \omega_l(\eta) = \omega_{\cl{M}}(r_{\Phi(\Theta(U))}^*(\id\tensor \bar{R}^*\tensor\id)(\id\tensor\id\tensor \eta)(\id\tensor \bar{R}\tensor\id)r_{\Phi(\Theta(U))}).
\]
Then we have
\begin{align*}
 \omega_l(\id) 
&= \omega_{\cl{M}}(r_{\Phi(\Theta(U))}^*(\id\tensor \bar{R}^*\bar{R}\tensor\id)r_{\Phi(\Theta(U))})\\
&= \omega_{\cl{M}}(r_{\Phi(\Theta(U))}^*r_{\Phi(\Theta(U))}\tensor \bar{R}^*\bar{R})\\
&= d(U)\tr_F(\id_F),\\
& \\
\omega_l(\bar{f}_U^*\bar{f}_U)
&= \omega_{\cl{M}}((r_{\Phi(\Theta(U))}^*\tensor\bar{R}^*)(\id\tensor f_Uf_U^*\tensor\id)(r_{\Phi(\Theta(U))}\tensor\bar{R}))\\
&= d(U)\tr_F(\id_F),\\
&\\
\omega_l(\bar{f}_U^*\bar{f}_U\bar{f}_U^*\bar{f}_U)
&= \omega_{\cl{M}}((r_{\Phi(\Theta(U))}^*\tensor\bar{R}^*)(\id\tensor f_U\tensor\id)(\id\tensor \id\tensor \bar{f}_U\bar{f}_U^*)\\
&\hspace{5cm}(\id\tensor f_U^*\tensor\id)(r_{\Phi(\Theta(U))}\tensor\bar{R}))\\
&= \omega_{\cl{M}}((r_{\Phi(\Theta(U))}^*\tensor\bar{R}^*)(\id\tensor f_U\tensor\id)(f_{\bar{U}}f_{\bar{U}}^*\tensor \bar{f}_U\bar{f}_U^*)\\
&\hspace{5cm}(\id\tensor f_U^*\tensor\id)(r_{\Phi(\Theta(U))}\tensor\bar{R}))\\
&= \omega_{\cl{M}}(\bar{R}^*(\id\tensor r_{\Phi(\Theta(U))}^*\tensor\id)(\id\tensor\id\tensor \bar{f}_U\bar{f}_U^*)(\id\tensor r_{\Phi(\Theta(U))}\tensor\id)\bar{R}) \\
&= \omega_{\cl{M}}((r_{\Phi(\Theta(U))}^*\tensor R^*)(\id\tensor \bar{f}_U\bar{f}_U^*\tensor\id)(r_{\Phi(\Theta(U))}\tensor R))\\
&= \omega_{\cl{M}}(r_{\Phi(\Theta(U))}^*(\id\tensor R^*\tensor\id)(\id\tensor\id\tensor f_U^*f_U)(\id\tensor R\tensor \id)r_{\Phi(\Theta(U))}) \\
&= \omega_l (\id)\\
&= d(U)\tr_F(\id_F).
\end{align*}
Hence $\omega_l((1 - \bar{f}_U^*\bar{f}_U)^2) = 0$ 
and $\bar{f}_U^*\bar{f}_U = 1$ since $\omega_l$ is faithful.

To show $\bar{f}_U\bar{f}_U^* = 1$, one can use the following
$\omega_r \in \End_{\mathrm{b}}(\Phi(\Theta(U))\tensor \bar{F})^*$:
\[
 \omega_r(\eta) = \omega_{\cl{M}}(\bar{r}_{\Phi(\Theta(U))}^*(\id\tensor R^*\tensor\id)(\eta\tensor\id\tensor\id)(\id\tensor R\tensor \id)\bar{r}_{\Phi(\Theta(U))}).
\]
Then a similar argument shows $\bar{f}_U\bar{f}_U^* = 1$.
\end{proof} 
 
\subsection{Imprimitivity-type result for the maximal Kac quantum subgroup}

In this subsection, we apply results in the previous subsection
to module categories arising from quantum homogeneous spaces.

At first, we give a characterization of a quantum homogeneous space
whose associated module category has a module trace.

\begin{prop}
Let $\G$ be a compact quantum group of Kac type and $A$ be 
a quantum homogeneous space of $\G$. Then the following conditions are 
equivalent:
\begin{enumerate}
 \item $\G\text{-}\mod^{\fin}_A$ has a $\rep^{\fin}\G$-module trace.
 \item There is a tracial state on $A$.
 \item The $\G$-invariant state on $A$ is tracial. 
\end{enumerate}
\end{prop}
\begin{proof}
The equivalence of (ii) and (iii) follows from the 
$\G$-invariance of $(\phi\tensor h)\alpha$ for a state $\phi$ on $A$.

We show (ii) $\Longrightarrow$ (i) by constructing a $\rep^{\fin}\G$-module
trace from a given tracial state $\tau$ on $A$. Let $E$ be 
a finitely generated $\G$-equivariant Hilbert $A$-module. 
Then there exists a finite dimensional unitary representation 
$\pi$ of $\G$ such that there is an isometry 
$V \in \cl{L}_A^{\G}(E,\cl{H}_{\pi}\tensor A)$.
Now we define a tracial state $\tr_E$ on $\cl{L}_A^{\G}(E)$ as follows:
\[
 \tr_E(T) = (\tr\tensor\tau)(VTV^*).
\]
Since $\tr\tensor \tau$ is tracial, this definition does not 
depend on the choice of $V$ and $\pi$. Moreover the corner trick shows that
$\tr_E$ satisfies (i) of Definition \ref{defn:module trace}. 
To prove the compatibility with the left action of $\rep^{\fin}\G$, 
one should note that the standard solution in $\rep^{\fin}\G$ is also 
standard in $\hilb^{\fin}$. Hence the partial trace 
$\map{\tr_{\pi}\tensor\id}{\cl{L}_A^{\G}(\cl{H}_{\pi}\tensor E)}{\cl{L}_A^{\G}(E)}$ coincides with $\tr\tensor\id$. This implies the compatibility.

Next we move to the proof of (i) $\Longrightarrow$ (ii). 
Fix a $\rep^{\fin}\G$-module trace $\{\tr_E\}_E$ 
on $\G\text{-}\mod^{\fin}_A$ with $\tr_A(1) = 1$.
By using the spectral decomposition of the algebraic core $\cl{A}$ of $A$,
the $\G$-invariant state $\phi$ on $A$ can be seen as the projection from
$\cl{A}$ to $\cl{A}_{\mbf{1}_{\G}}$. We show that this is tracial 
by using $\{\tr_E\}_E$. Since the spectral decomposition of $\cl{A}$ is 
orthogonal with respect to the inner product coming from $\phi$, 
it suffices to show that the $\cl{A}_{\mbf{1}_\G}$-components of 
$(T\tensor\xi)^*(S\tensor\eta)$ and $(S\tensor \eta)(T\tensor \xi)^*$ 
are same for $T\tensor \xi, S\tensor \eta \in \cl{A}_{\pi}$.
For the first one, its $\cl{A}_{\mbf{1}_{\G}}$-component is calculated
as follows:
\begin{align*}
&\frac{\ip{\xi,\eta}}{\dim\pi}
(R_{\pi}^*\tensor\id_A)(\id_{\bar{\pi}}\tensor T^*S)(R_{\pi}\tensor\id_A)\\
&=\frac{\ip{\xi,\eta}}{\dim\pi}
\tr_A((R_{\pi}^*\tensor\id_A)(\id_{\bar{\pi}}\tensor T^*S)(R_{\pi}\tensor\id_A))1_A\\
&= \frac{\ip{\xi,\eta}}{\dim\pi}\tr_{\cl{H}_{\pi}\tensor A}(T^*S)1_A \\
&= \frac{\ip{\xi,\eta}}{\dim\pi}\tr_A(ST^*)1_A \\
&= \frac{\ip{\xi,\eta}}{\dim\pi}ST^*.
\end{align*}
The last formula is precisely the $\cl{A}_{\mbf{1}_{\G}}$-component of 
the second one. Now we get the conclusion.
\end{proof}

Let $\G$ be 
a compact quantum group and $\K$ be a
its maximal Kac quantum subgroup. The restriction functor from
$\rep^{\fin}\G$ to $\rep^{\fin}\K$ is denoted by $\Theta$.
Fix a set $\irr\G$ of representatives of all equivalence classes of 
irreducible representations of $\G$.

For $\pi,\rho \in \rep^{\fin}\G$, we defines $L_0(\pi,\rho)$ as a subset of 
$\mathrm{Hom}_{\K}(\pi|_{\K},\rho|_{\K})$ consisting of 
intertwiners of the following form:
\[
 T^*\circ(a_{\Theta,\sigma_1}^{n_1}\tensor a_{\Theta,\sigma_2}^{n_2}\tensor\cdots\tensor a_{\Theta,\sigma_k}^{n_k})\circ S
\]
where $k \in \Z_{>0},\,n_i \in \frac{1}{2}\Z,\,\sigma_i \in \irr\G,\,T \in \mathrm{Hom}_{\G}(\rho,\sigma_1\tensor\sigma_2\tensor\cdots\tensor\sigma_k),\,
S \in \mathrm{Hom}_{\G}(\pi,\sigma_1\tensor\sigma_2\tensor\cdots\tensor\sigma_k)$. 
Then we define a subspace $L(\pi,\rho)$ of $\mathrm{Hom}_{\K}(\pi|_{\K},\rho|_{\K})$ as the linear span of intertwiners of the form 
$T_lT_{l - 1}\cdots T_1$ where $T_j \in L_0(\sigma_{j - 1},\sigma_j)$, $
\sigma_0 = \pi,\,\sigma_l = \rho,\,\sigma_j \in \irr\G$ for $1 \le j \le l - 1$.

\begin{lemm}
For any $\pi,\rho \in \rep^{\fin}\G$, we have 
$L(\pi,\rho) = \mathrm{Hom}_{\K}(\pi|_{\K},\rho|_{\K})$.
\end{lemm}
\begin{proof}
Consider a \cstar-category with $\obj \rep^{\fin}\G$ as the collection 
of object and $L(\pi,\rho)$ as the morphism set from $\pi$ to $\rho$. 
Let $\cl{L}$ be its Karoubi envelope. Since
$T \tensor S \in L(\pi\tensor\pi',\rho\tensor\rho')$ for any 
$T \in L(\pi,\rho)$ and $S \in L(\pi',\rho')$, $\cl{L}$ becomes a 
\cstar-tensor category. Moreover the restriction functor from 
$\rep^{\fin}\G$ to $\rep^{\fin}\K$ induces a \cstar-tensor
functor from $\rep^{\fin}\G$ to $\cl{L}$. In particular
$\pi \in \rep^{\fin}\G$ is rigid in $\cl{L}$, hence $\cl{L}$ is rigid.

Now, by composing the inclusion $\cl{L} \longrightarrow \rep^{\fin}\K$ 
and the fiber functor $\rep^{\fin}\K \longrightarrow \hilb^{\fin}$, 
we obtain a fiber functor $\cl{L}\longrightarrow \hilb^{\fin}$.
Then the Woronowicz's Tannaka-Krein duality ({\cite[Theorem 2.3.2]{NT13}}) 
implies that there is a compact qunatum group $\H$ and morphisms
$\K \longrightarrow \H$ and $\H \longrightarrow \G$ which correpond
to $\cl{L}$ and the functors. Since every
object of $\cl{L}$ appears as a direct summand of the image of
 $\pi \in \rep^{\fin}\K$, $\H$ can be considered as a quantum subroup of
$\G$ which contains $\K$.

On the other hand,
$(\pi,\bar{\pi},(\id\tensor a_{\Theta,\pi}^{1/2})R_{\pi}, (a_{\Theta,\pi}^{-1/2}\tensor\id)\bar{R}_{\pi})$ is a solution of the conjugate equation
in $\cl{L}$. Moreover it satisfies
\[
 \nor{(\id\tensor a_{\Theta,\pi}^{1/2})R_{\pi}}
=\nor{(a_{\Theta,\pi}^{-1/2}\tensor\id)\bar{R}_{\pi}}
=\dim \cl{H}_{\pi}.
\]
This implies that the fiber functor of $\rep^{\fin}\K = \cl{L}$ is 
dimension-preserving, hence $\H$ is of Kac type. Then $\H$ must be
contained in $\K$, so we have $\H = \K$. Hence 
the dimension of $L(\pi,\rho)$ and $\mathrm{Hom}_{\K}(\pi|_{\K},\rho|_{\K})$
is same. Now we get the conclusion.
\end{proof}

\begin{thrm} \label{thrm:main1}
Let $A$ be a quantum homogeneous space of $\K$ and $\tilde{A}$ be 
its induced $\G$-\cstar-algebra. 
If $A$ has a tracial state and
$\irr \K\text{-}\mod^{\fin}_A$ is finite,  
$\G\text{-}\corr^{\rfin}_{\tilde{A}}$ is rigid and
$\ind_{\K}^{\G}$ gives an equivalence 
$\K\text{-}\corr^{\rfin}_A\cong\G\text{-}\corr^{\rfin}_{\tilde{A}}$ 
as \cstar-tensor categories. 
\end{thrm}
\begin{proof}
Set $\cl{M} = \K\text{-}\mod^{\fin}_A$. Then $\cl{M}$ is a tracial 
$\rep^{\fin}\K$-module category with $\abs{\irr\cl{M}} < \infty$.

Since we have Theorem \ref{thrm:comparison},
it suffices to show the canonical inclusion from
$[\cl{M},\cl{M}]_{\mathrm{b}}^{\rep^{\fin}\K}$ to 
$[\cl{M},\cl{M}]_{\mathrm{b}}^{\rep^{\fin}\G}$
is an equivalence. 
It can be easily seen that this inclusion is fully faithful, hence 
the remaining part is the essential surjectivity.

Take $(F,f)\in [\cl{M},\cl{M}]^{\rep^{\fin}\G}_{\mathrm{b}}$.
Theorem \ref{thrm:finite descent} and 
Proposition \ref{prop:abstract descent} assert that we have
\[
 (\Phi(a_{\Theta,\pi})\tensor\id_F)f_{\pi} = 
f_{\pi}(\id_F\tensor \Phi(a_{\Theta,\pi}))
\] 
for any $\pi \in \rep^{\fin}\G$, where 
$\map{(\Phi,\phi)}{\rep^{\fin}\K}{[\cl{M},\cl{M}]_{\mathrm{b}}}$
is the canonical \cstar-tensor functor.

By using the braiding equation on $u$ and the previous lemma,
we also have 
\[
 (\Phi(T)\tensor\id_F)f_{\pi} = f_{\pi}(\id_F\tensor\Phi(T))
\]
for any $\pi\in \rep^{\fin}\G$ and 
any $T \in \mathrm{Hom}_{\K}(\pi|_{\K},\pi|_{\K})$.
Hence we can define $f_{\rho}$ for $\rho \in \rep^{\fin}\K$ as follows:
\[
 \tilde{f}_{\rho} = (\Phi(V)^*\tensor\id_F)f_{\pi}(\id_F\tensor \Phi(V))
\]
where $\pi \in \rep^{\fin}\G$ and $V \in \mathrm{Hom}_{\K}(\rho,\pi|_{\K})$
is an isometry. By the above equality, the definition of $\tilde{f}_{\rho}$ 
does not depend on the choice of $\pi$ and $V$. It can be easily seen
that $(F,\tilde{f})$ defines a $\rep^{\fin}\K$-module functor,
which is $(F,f)$ as a $\rep^{\fin}\G$-module functor.
\end{proof} 
\begin{exam}
 Set $A = C(\K)$. Then the corresponding $\rep^{\fin}\K$-module category
is $\hilb^{\fin}$ with the action via the fiber functor. This $\rep^{\fin}\K$-module category
has a $\rep^{\fin}\K$-module trace and only one irreducible object,
hence we can apply Theorem \ref{thrm:main1}. Since the induction of 
$C(\K)$ is $C(\G)$, we can conclude that $\G\text{-}\corr^{\rfin}_{C(\G)}$ 
is rigid and we have an equivalence of \cstar-tensor 
categories between $\K\text{-}\corr^{\rfin}_{C(\K)}$ and
$\G\text{-}\corr^{\rfin}_{C(\G)}$. 
On the other hand, we have 
$\K\text{-}\corr^{\rfin}_{C(\K)}\cong \rep^{\fin}\cl{O}(\K)$ and 
$\G\text{-}\corr^{\rfin}_{C(\G)} \cong \rep^{\fin}\cl{O}(\G)$ 
by Proposition \ref{prop:corr over C(G)}). Combining these equivalences,
we can see that $\rep^{\fin}\cl{O}(\G)$ is rigid and
$\rep^{\fin}\cl{O}(\K) \cong \rep^{\fin}\cl{O}(\G)$.
A direct calculation shows that this equivalence is induced by the 
canonical morphism $\map{q}{\cl{O}(\G)}{\cl{O}(\K)}$.

By these obserbation,
Theorem \ref{thrm:main1} can be regarded as a generalization of 
{\cite[Theorem 2.3]{CKS}} and 
$\mathfrak{b}\hat{\K} \cong \mathfrak{b}\hat{\G}$, which is 
essentailly proved in {\cite[Section 4.3]{So05}}.
\end{exam}

For a quantum homogeneous space $A$ of $\G$, we define its 
\emph{Picard group} $\picard_{\G}(A)$
as a group of all equivalence classes of invertible objects of 
$\G\text{-}\mod^{\fin}_A$ (c.f {\cite[Corollary 2.5]{DC12}}). 
Its product is given by the tensor product.

By using Proposition \ref{prop:locally scaling} instead of Theorem
\ref{thrm:finite descent} in the proof of Theorem \ref{thrm:main1}, 
we can obtain the following theorem.

\begin{thrm} \label{thrm:picard}
Let $\hat{\Gamma}$ be a cocommutative quantum subgroup of $\K$.
Then for any quantum homogeneous space $A$ of $\hat{\Gamma}$,
the induction functor gives a group isomorphism 
$\picard_{\K}(\ind_{\hat{\Gamma}}^{\K} A) \cong \picard_{\G}(\ind_{\hat{\Gamma}}^{\G}A)$.
\end{thrm}

Moreover we can also see that the induction functor gives an 
isomorphism between automorphism groups.
For a quantum homogeneous space $A$ of $\G$, the group
of $\G$-equivariant automorphisms of $A$ is denoted by $\auto_{\G}(A)$.
If $\theta \in \auto_{\G}(A)$ is given, we can construct an invertible
$\G$-equivariant correspondence $A_{\theta}$
over $A$ defined as follows:
\begin{itemize}
 \item As a $\G$-equivariant Hilbert $A$-module, $A_{\theta} = A$.
 \item The left multiplication of $x \in A$ on $A_{\theta}$ is given
by the left multiplication of $\theta^{-1}(x)$.
\end{itemize}

Then we have $A_{\theta}\tensor_A A_{\theta'} \cong A_{\theta\theta'}$,
so we have a group homomorphism $\theta \in \auto_{\G}(A)
\longmapsto [A_{\theta}] \in \picard_{\G}(A)$. This is injective
since $A^{\G} = \C1_A$.

\begin{lemm}
Let $M$ be an invertible $\G$-equivariant correspondence over $A$.
Then there is $\theta \in \auto_{\G}(A)$ such that $M \cong A_{\theta}$
if and only if $M$ has a non-zero $\G$-invariant vector.
\end{lemm}
\begin{proof}
 It is trivial that there exists a non-zero $\G$-invariant vector in
$A_{\theta}$. Hence it suffices to show the existence of a non-zero
$\G$-invariant vector implies the existence of $\theta$.
Since $M$ is irreducible as a $\G$-equivariant Hilbert $A$-module,
it must be isomorphic to $A$. Hence there is a left $A$-action on
$A$ by which $M$ and $A$ are isomorphic as $\G$-equivariant correspondence
over $A$. But such a correspondence must be of the form $A_{\theta}$
where $\map{\theta}{A}{A}$ is a $\G$-equivariant \star-homomorphism.
Since $M$ is invertible, this $\theta$ must be an automorphism.
Now we get the conclusion.
\end{proof}

\begin{thrm} \label{thrm:auto}
In the same setting as Theorem \ref{thrm:picard}, the induction
of equivariant automorphism gives a group isomorphism 
$\auto_{\K}(\ind_{\hat{\Gamma}}^{\K}A) \cong \auto_{\G}(\ind_{\hat{\Gamma}}^{\G}A)$.
\end{thrm}
\begin{proof}
One should note that $\ind_{\K}^{\G} E$ has a non-zero
$\G$-invariant vector if and only if $E$ has a non-zero $\K$-invariant 
vector for any $E \in \K\text{-}\mod^{\fin}_A$.

Then the statement follows from Theorem \ref{thrm:picard} and the previous
lemma.
\end{proof}

\begin{exam}
Let $Q \in GL_n(\C)$ be a positive invertible matrix with 
multiplicity-free eigenvalues. Then the maximal Kac quantum subgroup of 
the free unitary quantum group $U_Q^+$ ({\cite[Theorem 1.3]{DW96}}) is 
isomorphic to the discrete dual of the free group 
$F_n$ ({\cite[Theorem 2.6]{DFS}}). In particular
it is cocommutative, hence we have group isomorphisms 
\[
 \picard_{\hat{F_n}}(A) \cong \picard_{U_Q^+}(\ind_{\hat{F_n}}^{U_Q^+}A),
\quad \auto_{\hat{F_n}}(A) \cong \auto_{U_Q^+}(\ind_{\hat{F_n}}^{U_Q^+}A)
\]
for any quantum homogeneous space $A$ of $\hat{F_n}$.
\end{exam}

\section{Classification results for $G_q$}
\subsection{Imprimitivity-type results}

Let $G$ be a simply-connected compact Lie group and $T$ be its
maximal torus. The Weyl group of $G$ is denoted by $W$.
Consider the Drinfeld-Jimbo deformation $G_q$ for a fixed 
$q \in (-1,1)\setminus\{0\}$.
We use the symbol $\eps_t$ for the character on $C(G_q)$ defined as the
evaluation at $t \in T$.

Let us recall the representation theories of $C(G_q)$ and 
$C(T\backslash G_q)$, developed in \cite{DS99,NT12,So91}.

For each $w \in W$, we have an irreducible \star-representation 
$(\pi_w,\cl{H}_w)$ of $C(G_q)$. It coincides with the counit 
$\eps$ if $w = e$, and is infinite dimensional otherwise. 
By using these representations, $\hat{C(G_q)}$ can be identified 
with $W\times T$ as a set via 
$(w,t) \in W\times T \longmapsto (\eps_t\tensor \pi_w)\Delta$.
If we consider a Borel structure on $W\times T$ induced from $\hat{C(G_q)}$
via this identification, each $\{w\} \times T \subset W\times T$ is a 
Borel subset of $\hat{C(G_q)}$.

For $C(T\backslash G_q)$, the restriction of $\pi_w$ on 
$C(T\backslash G_q)$ is still irreducible and 
$w \in W\longmapsto \pi_w|_{C(T\backslash G_q)}$ is a bijection.
The induced Borel structure of $W$ is discrete.

The following proposition is what we would like to use in this subsection

\begin{prop}
The double dual $C(T\backslash G_q)^{**}$ is isomorphic to
$\prod_{w \in W} B(\cl{H}_w)$. Moreover, 
its center is contained in $Z(C(G_q)^{**})$.
\end{prop}
\begin{proof}
One shoule note that $C(T \backslash G_q)$ is separable type I.
Hence we can disintegrate any separable \star-representation of 
$C(T\backslash G_q)$ on $\hat{C(T\backslash G_q)} \cong W$. Since
$W$ is discrete as a Borel space, this implies that a 
separable \star-representation admits an irreducible decomposition.
By using a decomposition into cyclic representations, we also have
an irreducible decomposition for a non-separable \star-representation.
This implies the former half of the statement.

Let $p_w \in C(T\backslash G_q)^{**}$ be the projection corresponding to 
$B(\cl{H}_w)$. To show the latter half of the statement,
it suffices to show $\pi(p_w) \in \pi(C(G_q))'$ for any separable
\star-representation $(\pi,\cl{H})$ of $C(G_q)$.
Since $C(G_q)$ is type I, we can disintegrate $(\pi,\cl{H})$
on $\hat{C(G_q)}$. Then each $w \in W$ defines a projection 
$P_w \in \pi(C(G_q))'$ corresponding to the characteristic function of 
$\{w\}\times T \subset W\times T \cong \hat{C(G_q)}$. 

This $P_w$ must be $\pi(p_w)$, since we have
$(\eps_t\tensor \pi_w)\Delta|_{C(T\backslash G_q)} = \pi_w$ for any
$t \in T$. In particular we have $\pi(p_w) = P_w \in \pi(C(G_q))'$.
\end{proof}
\begin{rema}
Let $\{p_w\}_{w \in W}$ as in the proof. Then the canonical map
from $C(G_q)^{**}$ to $C(T)^{**}$ factors though
$x \in C(G_q)^{**}\longmapsto p_ex \in p_eC(G_q)^{{**}}$. Moreover
we have $p_eC(G_q)^{**} \cong C(T)^{**}$.
\end{rema}

\begin{coro} \label{coro:tracial descent}
Let $A$ be a $T$-\cstar-algebra and $\tilde{A}$ be its 
induced $G_q$-\cstar-algebra. For any tracial positive linear functional
$\tau $on $\tilde{A}$, there is a tracial positive linear functional
$\tau'$ on $A$ such that 
$\tau = \tau'\circ (\id\tensor \eps)|_{\tilde{A}}$.
\end{coro}
\begin{proof}
Take a tracial positive linear functional $\tau$ on $\tilde{A}$.
If $\tau = 0$, there is nothing to prove. Hence we may assume $\tau \neq 0$,
moreover $\tau$ is a tracial state by taking a normalization.

Now consider the GNS triple $(\pi_{\tau},\cl{H}_{\tau},\xi_{\tau})$
with respect to $\tau$. Then $\pi_{\tau}(\tilde{A})''$ is a
finite von Neumann algebra. In particular 
$\pi_{\tau}(C(T\backslash G_q))''$ is finite, 
hence $\pi_{\tau}(p_e) = 1$ and $\pi_{\tau}(p_w) = 0$ for $w \neq e$.

Take a state $\phi$ on $A\tensor C(G_q)^*$ such that 
$\phi|_{\tilde{A}} = \tau$ and its GNS triple 
$(\pi_{\phi},\cl{H}_{\phi},\xi_{\phi})$. Then we have an
isometry $\map{V}{\cl{H}_{\tau}}{\cl{H}_{\phi}}$ such that
$V\pi_{\tau}(x)\xi_{\tau} = \pi_{\phi}(x)\tau$ for every $x \in \tilde{A}$.
Hence, for any $w \in W\setminus\{e\}$, we have
\[
 \pi_{\phi}(1\tensor p_w)\xi_{\phi} = V\pi_{\tau}(p_w)\xi_{\tau} = 0.
\]
Since $1\tensor p_w$ is central in $(A\tensor C(G_q))^{**}$ and $\xi_{\phi}$
is cyclic, this implies $\pi_{\phi}(p_w) = 0$ for $w \neq e$.
Then the remark after Proposition 5.1 implies that $\pi_{\phi}$ factors 
though $A\tensor C(G_q)\longrightarrow A\tensor C(T)$, in particular
$\pi_{\tau}$ factors through $\tilde{A}\longrightarrow \ind_T^TA$.
Since $(\id\tensor\eps)|_{\ind_T^TA}$ gives an isomorphism 
$\ind_T^T A \cong A$, our statement follows. 
\end{proof}

% For a $T$-\cstar-algebra $A$, we define its 
% \emph{tracial quotient} $A_{\tra}$ as a quotient of 
% $A$ by the following ideal $J_A$:
% \[
%  J_A = \{x \in A\mid \tau(x^*x) = 0\text{ for any tracial positive linear functional }\tau \in A^*\}.
% \]
% Then $A_{\tra}$ is a $T$-\cstar-algebra. For a $G_q$-\cstar-algebra,
% we also defines its tracial quotient by considering the restricted
% action of $T$. 

The following is a main theorem of this subsection.

\begin{thrm} \label{thrm:hom descent}
Let $A$ and $B$ be $T$-\cstar-algebras. If $B$ has a faithful family
of tracial states,
any $G_q$-equivariant \star-homomorphism from $\ind_T^{G_q}A$ to 
$\ind_T^{G_q}B$ is induced from a $T$-equivariant \star-homomorphism
from $A$ to $B$.
\end{thrm}
\begin{proof}
 Let $\map{\phi}{\ind_T^{G_q}A}{\ind_T^{G_q}B}$ be 
an $G_q$-equivariant \star-homomorphism. For any tracial state $\tau$ on 
$B$, $\tau\circ(\id\tensor\eps)\circ\phi$ is also a tracial state on 
$\ind_T^{G_q}B$. Hence it factors through $A$. Since $B$ has
a faithful family of tracial states, the map $(\id\tensor \eps)\circ\phi$
also factors through $A$, i.e. we have a \star-homomorphism 
$\map{\phi_0}{A}{B}$ such that 
$(\id\tensor\eps)\circ \phi = \phi_0\circ(\id\tensor\eps)$.
The $T$-equivariance of $\phi_0$ follows from the $G_q$-equivariance of 
$\phi$.

We would like to show that $\phi_0$ is the required homomorphism.
For an element $a \in \ind_T^{G_q}A$, we have
\begin{align*}
 \phi(a) 
&= (\id\tensor\eps\tensor\id)(\tilde{\beta}(\phi(a)))\\
&= ((\id\tensor\eps)\circ\phi\tensor\id)(\tilde{\alpha}(a))\\
&= (\phi_0\circ (\id\tensor\eps)\tensor\id)(\tilde{\alpha}(a))\\
&= (\phi_0\tensor\id)(a).
\end{align*}
Hence we have $\phi = \ind_T^{G_q}\phi_0$.
\end{proof}
\begin{coro} \label{coro:descent for G_q}
 Let $A,B$ be quantum homogeneous spaces of $T$ and
$\tilde{A},\tilde{B}$ be its induced $G_q$-\cstar-algebras.
Then $\map{\ind_T^{G_q}}{T\text{-}\corr^{\rfin}_{A,B}}{G_q\text{-}\corr^{\rfin}_{\tilde{A},\tilde{B}}}$ is an equivalence of \cstar-categories.
\end{coro}
\begin{proof}
What we have to show is the essential surjectivity of $\ind_T^{G_q}$.
Take $\tilde{M} \in G_q\text{-}\corr^{\rfin}_{\tilde{A},\tilde{B}}$.
By Theorem \ref{thrm:module categorical induction}, we can take $M \in T\text{-}\mod^{\fin}_B$ such that $\tilde{M} \cong \ind_T^{G_q}M$ as a 
$G_q$-equivariant Hilbert $\tilde{B}$-module. Then the right action of
$\tilde{A}$ on $\tilde{M}$ induces a $G_q$-equivariant \star-homomorphism
from $\tilde{A}$ to 
$\cl{L}_{\tilde{B}}(\tilde{M})\cong \ind_T^{G_q}\cl{L}_B(M)$.
If we can show that $\cl{L}_B(M)$ has a faithful trace, the previous
theorem implies that this \star-homomorphism is induced by a $T$-equivariant
\star-homomorphism from $A$ to $\cl{L}_B(M)$. This defines a 
$T$-equivariant $(A,B)$-correspondence $M$ and we have
$\ind_T^{G_q}M \cong \tilde{M}$.

Now we show the existence of a faithful trace on $\cl{L}_B(M)$. By taking
an isometry from $M$ to $\cl{H}_{\pi}\tensor B$ for some 
$\pi \in \rep^{\fin}T$, we may assume $M = \cl{H}_{\pi}\tensor B$.
Then $\cl{L}_B(\cl{H}_{\pi}\tensor B) \cong B(H_{\pi})\tensor B$,
and the existence follows from the finiteness theorem 
{\cite[Theorem 4.1]{HKLS}}, which asserts that there is 
a faithful trace on $B$.
\end{proof}
\begin{coro} \label{coro:automatically induced}
 Let $\tilde{A}$ be a quantum homogeneous space of $G_q$.
If $\tilde{A}$ contains $C(T\backslash G_q)$ as a unital 
$G_q$-\cstar-subalgebra and has a tracial state, it must be 
induced from a quantum homogeneous space of $T$.
\end{coro}
\begin{proof}
Let $J$ be an ideal of $\tilde{A}$ defined as follows:
\[
 J = \{x \in \tilde{A}\mid \tau(x^*x) = 0\text{ for any tracial state }\tau \in \tilde{A}^*\}.
\]
By our assumption, this is a $T$-invariant 
proper closed ideal of $\tilde{A}$.

Set $A = \tilde{A}/J$. The quotient map from $\tilde{A}$ to $A$ is
denoted by $q$. The $T$-invariance of $J$ allows us to induce
a $T$-action on $A$. We will show that $(q\tensor\id)\tilde{\alpha}$
gives a $G_q$-equivariant \star-isomorphism from $\tilde{A}$ 
to $\ind_T^{G_q}A$. The well-definedness and $G_q$-equivariance
of this map are easy to see. Since $\tilde{A}^{G_q} = \C1_{\tilde{A}}$,
the injectivity is automatic. Hence it suffices to show the surjectivity.
By Theorem \ref{thrm:hom descent}, the restriction 
$(q\tensor\id)\tilde{\alpha}|_{C(T\backslash G_q)}$ is induced from 
$\C \subset A$ i.e. coincides with the canonical inclusion 
$C(T\backslash G_q) \subset \ind_T^{G_q}A$. Hence the image of this 
inclusion is contained in the image of $(q\tensor\id)\tilde{\alpha}$. 
On the other hand, for any $a \in \tilde{A}$ and $x \in C(G_q)$, we have
\begin{align*}
 P_A((q\tensor\id)\tilde{\alpha}(a)(1\tensor x))
&= (q\tensor\id)\tilde{\alpha}(a)P_A(1\tensor x) \\
&= (q\tensor\id)\tilde{\alpha}(a)(1\tensor (h_T\circ q_T\tensor \id)\Delta(x)),
\end{align*}
where $P_A$ is the expectation as in Lemma \ref{lemm:expectation onto ind},
$h_T$ is the Haar state on $C(T)$ and $q_T$ is the canonical
quotient map from $C(G_q)$ to $C(T)$.
Since $(h_T\circ q_T \tensor \id)\Delta(x)$ is an element of 
$C(T\backslash G_q)$,
we can conclude the surjectivity of $(q\tensor\id)\tilde{\alpha}$ from
this equality.
\end{proof}
\begin{coro}
 Let $\tilde{A}\subset \tilde{B}$ be a finite quantum 
$G_q$-covering space over a quantum homogeneous space. 
If $\tilde{A}$ contains $C(T\backslash G_q)$ and has
a tracial state, $\tilde{A} \subset \tilde{B}$ is isomorphic to
$\ind_T^{G_q}A \subset \ind_T^{G_q} B$,
where $A \subset B$ is a finite quantum $T$-covering space.
\end{coro}

\subsection{Discrete quantum subgroups}

In this subsection, we develop some applications to discrete quantum groups
and give a classification of finite index discrete quantum subgroup 
of $\hat{G_q}$.

Let $\G$ be a compact quantum group and $\H$ be its quotient.
Then we have an inclusion $C(\H) \subset C(\G)$. In this case we have
a unique $\G$-expectation $\map{E}{C(\G)}{C(\H)}$,
namely the Haar state-preserving conditional expectation.
  
\begin{defn}
Let $\hat{\G}$ be a discrete quantum group and $\hat{\H}$ be a quantum 
subgroup of $\G$. The scalar index of $\map{E}{C(\G)}{C(\H)}$
is called the quantum group index and denoted by $[\hat{\G}\colon\hat{\H}]$.
The quantum subgroup $\hat{\H}$ is said to be 
\emph{finite index} if $[\hat{\G}\colon\hat{\H}]$ is finite.
\end{defn}
\begin{rema}
The index $[\hat{\G}\colon\hat{\H}]$ must be either 
a positive integer or $\infty$. This can be seen from
Corollary \ref{coro:integer degree} if one consider
the basic construction of $C(\H) \subset C(\G)$.
\end{rema}

Beford giving a classification result of finite index quantum subgroup
of $\hat{G_q}$, we prepare the following lemma.
Let $\G$ be a compact quantum group and $\K$ be its maximal 
Kac quantum subgroup. The canonical map from $\cl{O}(\G)$ to $\cl{O}(\K)$
is denoted by $q$.

\begin{lemm}
Let $A$ be a right coideal of $\G$ and $\cl{A}$ be its algebraic core.
If the canonical inclusion $A \subset C(\G)$ is a finite quantum 
$\G$-covering space, we have $q^{-1}(q(\cl{A})) = \cl{A}$.
\end{lemm}
\begin{proof}
Let $\map{E}{C(\G)}{A}$ be a $\G$-expectation. Such a expectation is
unique and finite index by our assumption. 
Now consider the $\G$-equivariant Hilbert $A$-module $C(\G)_E$.
This gives a $\G$-equivariant imprimitivity bimodule betwee $A$ and 
$B = \cl{L}_A(C(\G)_E)$. In particular we have an equivalence 
$\G\text{-}\mod^{\fin}_A\cong \G\text{-}\mod^{\fin}_B$ of 
$\rep^{\fin}\G$-module category. Moreover, $\irr\G\text{-}\mod^{\fin}_A$ 
is finite since 
$\map{\tend\tensor_{C(\G)} C(\G)_E}{\G\text{-}\mod^{\fin}_{C(\G)}}{\G\text{-}\mod^{\fin}_A}$ is adjointable and $\irr\G\text{-}\mod^{\fin}_{C(\G)}$
consists of only one element.
 
On the other hand, since $C(\G) \subset B$ is a finite quantum $\G$-covering
space, Theorem \ref{thrm:main1} implies that there is a finite quantum
$\K$-covering space $C(\K) \subset B_0$ such that $C(\G) \subset B$
is isomorphic to $\ind_{\K}^{\G} C(\K) \subset \ind_{\K}^{\G} B_0$.
In particular we have an equivalence 
$\G\text{-}\mod^{\fin}_B\cong \K\text{-}\mod^{\fin}_{B_0}$ as 
$\rep^{\fin}\G$-module categories. Hence we have a structure of 
$\rep^{\fin}\K$-module category on $\G\text{-}\mod^{\fin}_A$.
Since $\irr\G\text{-}\mod^{\fin}_A$ is finite, Theorem 
\ref{thrm:finite descent} implies that $\tend\tensor_A \cl{O}(\G)$
is a $\rep^{\fin}\K$-module functor from $\G\text{-}\mod^{\fin}_A$
to $\G\text{-}\mod^{\fin}_{C(\G)}\cong \K\text{-}\mod^{\fin}_{C(\K)}$.

Take a quantum homogeneous space $A_0$ of $\K$ corresponding to
an irreducible object $A$ of a $\rep^{\fin}\K$-module category 
$\G\text{-}\mod^{\fin}_A$. If we fix an identification 
$\G\text{-}\mod^{\fin}_A \cong \K\text{-}\mod^{\fin}_{A_0}$,
we have a corresponding $\G$-equivariant isomorphism 
$A \cong \ind_{\K}^{\G}A_0$. Moreover $\tend\tensor_A \cl{O}(\G)$
induces a $\rep^{\fin}\K$-module functor from $\K\text{-}\mod^{\fin}_{A_0}$
to $\K\text{-}\mod^{\fin}_{C(\K)}$. By Theorem 
\ref{thrm:fundamental theorem}, we have a corresponding 
$M \in \K\text{-}\corr^{\rfin}_{A_0,C(\K)}$, which satisfies 
$A_0\tensor_{A_0}M \cong C(\K)$. Hence $M$ can be thought as $C(\K)$
as a $\K$-equivariant Hilbert $C(\K)$-module. Then we have a 
$\K$-equivariant \star-homomorphism $\map{\phi}{A_0}{C(\K)}$. 
By our construction, this map makes the following diagram commutative:
\[
 \xymatrix
{
\ind_{\K}^{\G}A_0 \ar[r]^-{\ind_{\K}^{\G}\phi} \ar[d]_-{\cong} & \ind_{\K}^{\G}C(\K) \ar[d]^-{\eps\tensor \id} \\
A \ar[r] & C(\G).
}
\]
This implies that $q(A) = \phi(A_0)$ and $q^{-1}(\phi(A_0)) = A$.
\end{proof}

Now we give a classification theorem of finite index discrete quantum subgroup of $\hat{G_q}$. Let $P$ (resp. $Q$) be the weight (resp. root) lattice.
We naturally identifies $P$ with the Pontrjagin dual of $T$. 

\begin{thrm} \label{thrm:classification of subgroup}
There is a one-to-one correspondence between the set of 
finite index discrete quantum subgroups of $\hat{G_q}$ and the set of 
subgroups of $P/Q$. 
\end{thrm}
\begin{proof}
Let $\map{q}{C(G_q)}{C(T)}$ be the canonical map.
Take a finite index discrete quantum subgroup $\hat{\H} \subset \hat{G_q}$.
The representation category $\rep^{\fin}\H$ can be thought as
a full subcategory of $\rep^{\fin} G_q$.

Consider the image $q(C(\H)) \subset C(T)$. This defines a compact 
quantum group $H$ if it is equipped with the restriction of 
the coproduct on $C(T)$. Then $H$ is a quotient compact group of 
$T$ and gives a full subcategory 
$\rep^{\fin}H$ of $\rep^{\fin}T$. The property $q(C(\H)) = C(H)$ implies 
that $\rep^{\fin} H$ consists of $\pi \in \rep^{\fin}T$ which appears
as a direct summand of $\rho|_T$ for some 
$\rho \in \rep^{\fin}\H \subset \rep^{\fin}G_q$. Combining this with
the highest weight theory, we can see that $\hat{H} \subset \hat{T} = P$
must contains the root lattice $Q$.

On the other hand, we have $q^{-1}(C(H)) = C(\H)$. This implies that
$\rep^{\fin}\H$ consists of $\pi \in \rep^{\fin}G_q$ such that 
$\pi|_T \in \rep^{\fin} H \subset \rep^{\fin} T$. Hence $\rep^{\fin}\H$
is the full subcategory of $\rep^{\fin}G_q$ consisting of 
$\pi \in \rep^{\fin}G_q$ whose weights are elements of $\hat{H}$.
These obserbations imply that $\hat{\H}\longmapsto \hat{\H}/Q$ is an injective map
from the set of finite index discrete quantum subgroups of $\hat{G_q}$
to the set of subgroups of $P/Q$. 

To show the surjectivity, take a subgroup $\Lambda \subset P$ containing
$Q$. Let $\cl{C}$ be the full subcategory of $\rep^{\fin}G_q$ consisting
of $\pi \in \rep^{\fin}G_q$ whose weights are elements of $\Lambda$.
Then this is a rigid \cstar-tensor category, hence we have a 
discrete quantum group $\hat{\H}$ of $\hat{G_q}$ such that 
$\rep^{\fin}\H \cong \cl{C}$. Moreover the inclusion 
$C(\H) \subset C(G_q)$ is induced from $C(\hat{\Lambda}) \subset C(T)$.
Since $Q \subset P$ is finite index, this inclusion has a
$T$-expectation with finite index. Hence $C(\H) \subset C(G_q)$ has
a $G_q$-expectation with finite index i.e. $\hat{\H}$ is finite index 
$\hat{G_q}$. Since the image of $\hat{\H}$ under the map defined in
the previous paragraph is $\Lambda$, we have the surjectivity.
\end{proof}

\appendix
\section{}
\subsection{Proof of Propostion \ref{prop:automatic continuity}.}
\label{subsection:A1}

\begin{lemm} \label{lemm:fixed points are non-degenerate}
Let $A$ be a $\G$-\cstar-algebra. Then its fixed point subalgebra
$A^{\G}$ is non-degenerate in $A$. Moreover, for any $a \in \cl{A}$,
we have $x \in A^{\G}$ and $b \in \cl{A}$ such that $a = xb$.
\end{lemm}
\begin{proof}
Let $\cl{A}$ be the algebraic core of $A$. For
the non-degeneracy of $A^{\G} \subset A$, it suffices to show that 
any $a \in \cl{A}$ can be approximated by an element of 
the form $xb$ with $x \in A^{\G}$ and $b \in \cl{A}$. 
Take an arbitrary $\eps > 0$ and set 
$\alpha(a) = \sum_{i = 1}^n a_{0}^i\tensor a_1^{i}$. Then we can find
$e \in A$ such that $\nor{ea_0^i - a_0^i} < \eps(n\nor{S(a_1^i)})^{-1}$ 
for $1 \le i \le n$. Now we have
\begin{align*}
 \alpha(e)(a\tensor 1) - a\tensor 1 
= \sum_{i = 1}^n \alpha(ea_0^i - a_0^i)(1\tensor S(a_1^i)).
\end{align*}
Hence we also have
\begin{align*}
 \nor{(\id\tensor h)(\alpha(e))a - a} 
 \le \nor{\alpha(e)(a\tensor 1) - a\tensor 1}
 \le \sum_{i = 1}^n \nor{ea_0^i - a_0^i}\nor{S(a_1^i)}
 < \eps.
\end{align*}
Since $(\id\tensor h)(\alpha(e))$ is in $A^{\G}$, this inequality
completes the proof of the first half.

Next we show tha latter half of our statement. Take an irreducible
representation $\pi$ of $\G$ and let $C(\G)_{\pi}$ be the linear
span of matrix coefficients of $\pi$. Moreover set $A_{\pi}$ as follows:
\[
 A_{\pi} = \{x \in A\mid \alpha(x) \in A\tensor C(\G)_{\pi}\}.
\]
Then the following properties hold as shown in {\cite[Section 3]{DC16}}:
\begin{enumerate}
 \item $\cl{A} = \bigoplus_{\pi \in \irr\G} A_{\pi}$.
 \item Each $A_{\pi}$ is a closed $A^{\G}$-subbimodule of $A$.
 \item The projection onto $A_{\pi}$ with respect to the decomposition in 
(i) extends to a continuous projection $\map{E_{\pi}}{A}{A_{\pi}}$.
\end{enumerate} 

By using (i), it can be seen that (ii) and (iii) hold for
$A_F = \sum_{\pi \in F} A_{\pi}$, where $F$ is a finite subset of $\irr\G$. 

Take an arbitrary $a \in \cl{A}$ and a finite subset 
$F \subset \irr\G$ such that $a \in A_F$. Since $A^{\G}$ is 
non-degenerate in $A$, an approximate unit
of $A^{\G}$ acts as an approximate unit of the left Banach $A^{\G}$-module
$A_F$. Then we can take $x \in A^{\G}$ and $b \in A_F$ such that
$a = xb$ by using the Cohen-Hewitt factorization theorem ({\cite[Chapter I, Section 11, Theorem 10]{BD73}}). Since
$A_F \subset \cl{A}$, this completes the proof. 
\end{proof}

\begin{prop}[Proposition \ref{prop:automatic continuity}]
Let $A$ and $B$ be $\G$-\cstar-algebras with the algebraic cores
$\cl{A}$ and $\cl{B}$ respectively. If $\map{\phi}{\cl{A}}{\cl{B}}$ is 
$\G$-equivariant and completely positive as a map from $\cl{A}$ to $B$, 
it extends to a $\G$-equivariant c.p.~map from $A$ to $B$.
\end{prop}
\begin{proof}
 By replacing $\phi$ by $\map{\phi\oplus\id}{\cl{A}}{\cl{B}\oplus\cl{A}}$,
we may assume $\phi$ is faithful in the sense that $\phi(x^*x) = 0$ if
and only if $x = 0$. 

By our assumption, we can define a $B$-valued 
semi-inner product on $\cl{A}\tensor B$ as follows:
\[
 \ip{a\tensor b, a'\tensor b'}_{B} = b^*\phi(a^*a')b'.
\]
Let $E$ be a Hilbert $B$-module obtained by taking a completion of 
$\cl{A}\tensor B$. We show that $a \in \cl{A}$ defines a adjointable
right $B$-module map $L_a$ on $E$ by $L_a(a'\tensor b') = aa'\tensor b'$.
Take an arbitrary element 
$\sum_{i = 1}^n a_i\tensor b_i \in \cl{A}\tensor B$.
Then Lemma \ref{lemm:fixed points are non-degenerate} implies that
we have a presentation $a_i = x_ia_i'$ with $x_i \in A^{\G}$ and 
$a_i' \in \cl{A}$. Since $A^{\G}$ is a \cstar-algebra,
for any $a \in A^{\G}$ we can take 
$X \in A^{\G}\tensor M_n(\C)$ such that
\[
 \mbf{x}^*a^*a\mbf{x} + X^*X = \nor{a}^2\mbf{x}^*\mbf{x},
\]
where $\mbf{x} = (x_1\,x_2\,\dots\,x_n)$. Then the complete positivity
of $\phi$ implies that
\begin{align*}
 \ip{\sum_{i = 1}^n aa_i\tensor b_i,\sum_{j = 1}^n aa_j\tensor b_j}_{B}
&= \sum_{i,j = 1}^n b_i^*\phi(a_i'^*x_i^*a^*ax_ja_j')b_j \\
&\le \nor{a}^2 \sum_{i,j = 1}^n b_i\phi(a_i'^*x_i^*x_ja_j')b_j \\
&= \nor{a}^2 \ip{\sum_{i = 1}^n a_i\tensor b_i, \sum_{j = 1}^n a_j\tensor b_j}_{B}. 
\end{align*}
Hence $L_a$ is well-defined and bounded. 

For general $a \in \cl{A}$, we use the argument in the proof of 
{\cite[Proposition 4.1]{DC16}}. Then we can take a family
$(v_i)_{i = 1}^n \subset \cl{A}$ such that $a = v_1$ and 
$\sum_i v_i^*v_i \in A^{\G}$, 
which gives an estimate from above as follows:
\[
 \ip{a\xi,a\xi}_{B} 
\le \ip{\xi,\sum_{i = 1}^nv_i^*v_i\xi}_{B}
\le \nor{\sum_{i = 1}^n v_i^*v_i}\ip{\xi,\xi}_{B}.
\]
This proves the boundedness of $L_a$. The adjointability can be
seen from the following:
\[
 \ip{aa_1\tensor b_1, a_2\tensor b_2}_{B}
= \ip{a_1\tensor b_1,a^*a_2\tensor b_2}_{B}.
\]
In particular $L_a^* = L_{a^*}$ holds, hence $a \longmapsto L_a$
defines a \star-homomorphism $\map{L}{\cl{A}}{\cl{L}_{B}(E)}$. 
This map is faithful since $\phi$ is faithful.

Next we show that there is a unitary 
$\map{V}{E\tensor_{\beta}(B\tensor C(\G))}{E\tensor C(\G)}$ which
satisfies the following for any $a \in \cl{A},\,b,c \in B$ and $x \in C(\G)$.
\[
 V((a\tensor b)\tensor (c\tensor x)) 
= \alpha(a)_{13}\beta(b)_{23}(1\tensor c\tensor x).
\]
Tha well-definedness and the boundedness of this map can be seen as 
follows:
\begin{align*}
& \ip{V((a\tensor b)\tensor (c\tensor x)), V((a'\tensor b')\tensor (c'\tensor x'))}_{B\tensor C(\G)}\\
&= (c^*\tensor x^*)\beta(b^*)(\phi\tensor\id)(\alpha(a^*a'))\beta(b')(c'\tensor x')\\
&= (c^*\tensor x^*)\beta(b^*\phi(a^*a')b')(c'\tensor x')\\
&= \ip{(a\tensor b)\tensor (c\tensor x), (a'\tensor b')\tensor (c'\tensor x')}_{B\tensor C(\G)}.
\end{align*}
In particular $V$ is an isometry, hence it suffices to show that
the range of $V$ is dense in $E\tensor C(\G)$. 
This follows from 
$\aspan \alpha(\cl{A})(\C\tensor \cl{O}(\G)) = \cl{A}\tensor\cl{O}(\G)$
and $\aspan \beta(\cl{B})(\C\tensor \cl{O}(\G)) = \cl{B}\tensor \cl{O}(\G)$.

Let $\map{\lambda}{C(\G)}{\cl{L}_{C(\G)}(C(\G))}$ be a \star-homomorphism
defined by the left multiplication. Then we have 
$V(L_a\tensor_{\beta} 1)V^* = (L\tensor \lambda)\alpha(a)$ 
for any $a \in \cl{A}$. This implies the following: 
\[
 \nor{L_a} 
= \nor{L_a\tensor_{\beta} 1} 
= \nor{(L\tensor \lambda)(\alpha(a))} 
= \nor{(L\tensor\id)(\alpha(a))}.
\]
We use the faithfullness of $\beta$ and $\lambda$ at the first and third
equality respectively. Since
$L$ is faithful, $a \in \cl{A} \longmapsto \nor{L_a}$ defines 
a \cstar-norm on $\cl{A}$. The above equality implies the completion
of $\cl{A}$ with respect to this norm has an action of $\G$ induced by
$\alpha$. Then {\cite[Proposition 4.4]{DY13}} implies that
this norm coincides with the original norm of $A$. Hence $L$ extends
to a \star-homomorphism $\map{L}{A}{\cl{L}_{B^{\G}}(E)}$.

Then we have the following inequality for $x,a \in \cl{A},\,b \in B$:
\begin{align*}
 \nor{b^*\phi(a^*xa)b} 
&= \nor{\ip{a\tensor b,L_x(a\tensor b)}_B} \\
&\le\nor{x}\nor{a\tensor b}^2 
\le\nor{x}\nor{b}^2\nor{\phi(a^*a)}.
\end{align*}
By using an approximate unit of $B$, we also have the following inequality
for $x,a \in \cl{A}$:
\[
 \nor{\phi(a^*xa)}\le \nor{x}\nor{\phi(a^*a)}.
\]
Hence $x \longmapsto \phi(a^*xa)$ is continuous for any $a \in \cl{A}$.
By using the polarization identity, we also see that
$\phi_{a,a'}\colon x \longmapsto \phi(axa')$ is continuous for 
any $a,a' \in \cl{A}$.

To show the continuity of $\phi$, take a null sequence 
$(x_n)_{n}$ in $\cl{A}$. Then we can find a sequence 
$(F_n)_{n}$ of finite subsets of $\irr\G$ such that
$x_n \in A_{F_n}$, here $A_{F_n}$ is the subspace of $A$ as in the proof
of Lemma \ref{lemm:fixed points are non-degenerate}. Set 
$M = c_0\text{-}\bigoplus_n A_{F_n}$. Then $A^{\G}$ acts on $M$ from 
the left and the right. Moreover an approximate unit of $A^{\G}$ acts 
as a bounded approximate unit on the both sides of $M$, since 
$A^{\G}$ is non-degenerate in $A$. Hence we can apply the 
Cohen-Hewitt factorization theorem ({\cite[Chapter I, Section 11, Theorem 10]{BD73}}) to $M$ as a $A^{\G}$-bimodule.

Now consider $(x_n)_n$ as an element $x \in M$. Then we can find
$a,\,a'$ and $y \in M$ such that $x = aya'$ i.e. we can find 
a null sequence $(y_n)_n \subset \cl{A}$ such that $x_n = ay_na'$ 
for all $n$. The continuity of $\phi_{a,a'}$ implies that
$\phi(x_n) = \phi(ay_na') = \phi_{a,a'}(y_n)$ converges to $0$ in norm.
Hence $\phi$ is continuous with respect to the norm-topology and 
extends to $A$. The complete positivity and $\G$-equivariance
follows from the corresponding property for $\phi$.
\end{proof}

\subsection{Proof of Proposition \ref{prop:finiteness}}
\label{subsection:A2}

\begin{prop}[Proposition \ref{prop:finiteness}]
Let $\map{E}{B}{A}$ be a conditional expectation. 
Then $E$ is with finite index if and only if 
it satisfies both of the following two conditions:
\begin{enumerate}
 \item The probabilistic index of $E$ is finite. 
 \item The Hilbert $A$-module $B_E$ can be decomposed into a direct sum of 
finitely generated projective Hilbert $A$-modules. 
\end{enumerate}
In this case the scalar index of $E$ coincides with $\nor{\Index{E}}$.
\end{prop}

\begin{proof}
Assume $E$ is with finite index and take a quasi-basis 
$(v_i)_{i = 1}^n \subset B$. Then we have a linear map 
$\map{V}{B_E}{A^n}$ defined by $V(x) = (E(v_i^*x))_{i = 1}^n$. Then
we have
\[
 \ip{V(y),V(x)}_A
= \sum_{i = 1}^n E(v_i^*y)^*E(v_i^*x)
= E\left(y^*\sum_{i = 1}^n v_iE(v_i^*x)\right)
=E(y^*x) = \ip{y,x}_A.
\]
Hence $V$ ia an isometrical $A$-module map, so $B_E$ is finitely
generated projective Hilbert $A$-module. The finiteness of 
the probabilistic index of $E$ follows from the inequality 
$\nor{\Index E}E(x) - x \ge 0$, which was shown in 
{\cite[Lemma 2.6.2.]{Wa90}}.
Moreover we can also see that $\nor{\Index E}$ is equal to or greater than
the scalar index of $E$ by replacing $E$ by 
$\map{E\otimes\id}{B\tensor M_n(\C)}{A\tensor M_n(\C)}$, where $n$ is 
an arbitrary positive integer.

Next we prove the converse direction. We firstly remark that the original
C*-norm on $B$ is equivalent to the $L^2$-norm with respect to $E$, hence
$B_E = B$ as $\C$-vector space. 

By the previous proposition, the scalar index $c$ of $E$ is finite. 
Then the product map on $B$ induces a bounded $B$-module map 
$\map{T}{B_E\tensor_A B}{B}$ with $\nor{T}^2 \le c$, since we have
\begin{align*}
 \ip{\sum_{i = 1}^n x_iy_i,\sum_{i = 1}^n x_iy_i}_B
&= {}^t\mbf{y}^*\mbf{x}^*\mbf{x}{}^t\mbf{y} \\
&\le c {}^t\mbf{y}^*(E\tensor\id )(\mbf{x}^*\mbf{x}){}^t\mbf{y}
=c \ip{\sum_{i = 1}^nx_i\tensor y_i,\sum_{i = 1}^n x_i\tensor y_i}_B,
\end{align*}
where $\mbf{x} = (x_1\,x_2\,\dots\,x_n),\,\mbf{y} = (y_1\,y_2\,\dots\,y_n)$.
Now we can find $\eta \in B_E\tensor_A B$ with $T\xi = \ip{\eta,\xi}_B$
for any $\xi \in B_E\tensor_A B$ since $B_E$ is a direc sum of 
finitely generated projective Hilbert $A$-module. Moreover we can find
sequences $\{x_i\}_{i = 1}^{\infty} \subset B_E$ and 
$\{y_i\}_{i = 1}^{\infty} \subset B$ such that 
$\sum_{i = 1}^n x_i\otimes y_i$ converges to $\eta$. Let $S_n$ be an 
$A$-module map on $B_E$ defined by 
\[
S_n(z) = \sum_{i = 1}^n y_i^*\ip{x_i,z}_A 
= \ip{\sum_{i = 1}^n x_i\tensor y_i, z\tensor 1}. 
\]
Then $S_k$ converges to $\id_B$ in the operator norm, hence 
$S_k$ is invertible for some $k$. This fact enables us to 
find $(u_i)_{i = 1}^n$ and $(v_i)_{i = 1}^n$ in $B$ such that 
$x = \sum_{i = 1}^n u_iE(v_i^*x)$ for any $x \in B$. Then we can find 
a quasi-basis of $E$ by {\cite[Lemma 2.1.6.]{Wa90}}.

For tha last statement, it suffices to show $\nor{\Index E} \le c$. But
this can be easily seen since $T^*(1_B) = \sum_{i = 1}^n v_i \tensor v_i^*$
 and $TT^*(1_B) = \Index E$.
\end{proof}

\vspace{10pt}
\noindent
{\bf Acknowlegements.}
The author appreciates to Yasuyuki Kawahigashi for helpful comments 
and pointing out mistakes and typos on this paper. He is also
grateful to Reiji Tomatsu for doing seminars with the author. 
This work was supported by Forefront Physics and Mathematics Program.

\end{document}